%% file: diffharmo.tex
\documentclass[11pt]{amsart}

\usepackage{amsmath}
\usepackage{amsfonts}
\usepackage{amssymb}
\usepackage{amsthm}
\usepackage{graphics}
\usepackage{amscd}
\usepackage{graphicx,epsfig}
\usepackage{color}
\usepackage[all]{xy}
\usepackage{url}

\DeclareMathOperator{\Div}{div}

\DeclareMathOperator{\Ric}{Ric}

\renewcommand{\epsilon}{\varepsilon}

\newcommand{\boD}{\mathcal{D}}

\newcommand{\boB}{\mathcal{B}}

\newcommand{\boS}{\mathcal{S}}
\newcommand{\boP}{\mathcal{P}}

\newcommand{\boK}{\mathcal{K}}

\newcommand{\R}{\mathbb{R}}

\newcommand{\C}{\mathbb{C}}

\renewcommand{\H}{\mathbb{H}}
\renewcommand{\S}{\mathbb{S}}

\newcommand{\la}{\langle}
\newcommand{\ra}{\rangle}
\newcommand{\eps}{\varepsilon}

\newcommand{\Ome}{\Omega}

\newtheorem{thm}{Theorem}
\newtheorem{defn}[thm]{Definition}
\newtheorem{cor}[thm]{Corollary}
\newtheorem{prop}[thm]{Proposition}
\newtheorem{lem}[thm]{Lemma}
\newcommand{\inter}[1]{\overset{\circ}{#1}}
\newcommand{\barre}[1]{\overline{#1}}
\renewcommand{\phi}{\varphi}

\newtheorem*{thm*}{Theorem}

\newtheorem*{claim*}{Claim}

\theoremstyle{remark}

\newtheorem{remarq}{Remark}
\newtheorem*{rem*}{Remark}

\newcounter{remark}

\newcounter{case}

\newcounter{construction}

\newcounter{fact}

\newcounter{step}

\newcommand{\bft}{\mathbf t}

\title[Minimal graphs and harmonic diffeomorphisms]{Minimal graphs over Riemannian
surfaces and harmonic diffeomorphisms}
\author{Laurent Mazet}
\address{Universit\'e Paris-Est, Laboratoire d'Analyse et de Math\'ematiques Appliqu\'ees (UMR 8050), UPEM, UPEC, CNRS, F-94010, Cr\'eteil, France}
\email{laurent.mazet@math.cnrs.fr}

\author{Magdalena Rodr\'iguez}
\address{Departamento de Geometra y Topologa, Universidad de Granada, Fuentenueva, 18071, Granada, Spain}
\email{magdarp@ugr.es}
\thanks{Research partially supported by the MCyT-FEDER research project MTM2014-52368-P}
          
\author{Harold Rosenberg}
\address{Instituto Nacional de Matematica Pura e Aplicada. Estrada Dona Castorina 110, Rio de Janeiro 22460-320. Brazil}
\email{rosen@impa.br }

\begin{document}
\maketitle

\begin{abstract}
We construct a parabolic entire minimal graph $S$ over a finite topology complete Riemannian surface $\Sigma$ of curvature $-1$ and infinite area (thus of non-parabolic conformal type). The vertical projection of this graph yields a harmonic diffeomorphism from $S$ onto $\Sigma$. The proof uses the theory of divergence lines to construct minimal graphs.

We also generalize a theorem of R.~Schoen. Let $g_1$ and $g_2$ be two complete metrics on a orientable surface $S$ with compact boundary and suppose 
$$\int_{S_r^2}K_{g_2}^-d\sigma_{g_2}\le C\ln(2+r)$$
for some $C>0$ and all $r>0$. If there is a harmonic diffeomorphism from $(S,g_1)$ to $(S,g_2)$, then $(S,g_1)$ is parabolic.
\end{abstract}


\section{Introduction}
Perhaps Bernstein proved the first global theorem concerning minimal
graphs: An entire minimal graph over the euclidean plane $\R^2$ is a
plane. 
This has had a great influence on minimal surface theory and partial
differential equations.  Among the many different proofs of
Bernstein's theorem that followed, that of Heinz~\cite{Hei} used harmonic
diffeomorphisms.  He proved there is no harmonic diffeomorphism from
the disk $\{ x^2 + y^2 < 1 \}$ onto $\R^2$ with the Euclidean metric.
The vertical projection of a minimal graph over a Riemannian manifold
is a harmonic diffeomorphism onto its image.  Thus Heinz concluded
that an entire minimal graph over $\R^2$ is necessarily conformally
the complex plane $\C$.  The Gauss map of the graph then defines a
holomorphic bounded function on $\C$, hence is constant, and the graph is
a plane.

Thus the existence of minimal graphs is intimately related to the
existence (or non existence) of harmonic diffeomorphisms.  Until the
last decade, the theory of minimal graphs over surfaces and harmonic
diffeomorphisms between surfaces developed considerably, yet
independently.  Before discussing some of these developments, we state
our main results.

In this paper we will construct an entire minimal parabolic graph
$\Sigma$ over any complete Riemannian surface $M$ of sectional curvature
$-1$, finite topology, and infinite area.  Thus we obtain a harmonic diffeomorphism of
$\Sigma$ onto $M$ (Theorem~\ref{th:harm}).  Parabolic here means the annular ends of $\Sigma$ are
conformally $\{z\in\C \ |\ 1\le |z|\}$, and infinite area implies
$M$ has at least one hyperbolic end: conformally $\{ 1 \leq |z| < c\}$ where $c<+\infty$.

In~\cite{Sch3}, Schoen proved that there is no harmonic diffeomorphism from
the unit disk onto a complete surface of non negative curvature; this is a generalization of Heinz result. We improve Schoen's result.  Let $S$ be an orientable surface with a compact
boundary. Let $g_1$ and $g_2$ be two complete Riemannian metrics on
$S$. Assume that there is a constant $C\geq 0$ such that, for any $r$,
\[
\int_{S_r^2}K_{g_2}^-d\sigma_{g_2}\le C\ln(2+r)
\]
where $K_{g_2}^-=\max\{0,-K_{g_2}\}$ and $S_r^2=\{p\in
S|d_{g_2}(p,\partial S))<r\}$.  If there is a harmonic diffeomorphism
$u :(S,g_1)\to (S,g_2)$, then $(S,g_1)$ is parabolic.

Let us now come back to the historical background of our work.

\subsubsection*{Minimal Graphs.}
Almost a century before Bernstein proved his theorem, Scherk
constructed many interesting minimal graphs over domains in $\R^2$.
The best known example is the graph of $\ln\frac{\cos(x)}{\cos(y)}$,
over the square $(-\pi/2, \pi/2)\times(-\pi/2, \pi/2)$ in $\R^2$,
taking $+\infty$ and $-\infty$ values over opposite sides of the
square.  The graph is bounded by the four vertical lines over the
vertices of the square and extends to a complete doubly periodic
minimal surface in $\R^3$ by successive rotations by $\pi$ about the
vertical lines.

Jenkins and Serrin~\cite{JeSe} found necessary and sufficient geometric conditions
on compact domains in $\R^2$ bounded by piecewise smooth arcs, to find minimal
graphs over the domain taking prescribed boundary values (perhaps
infinite) on the boundary arcs.

There have been many generalizations of their theorem to domains in
Riemannian surfaces.  Of interest to us here is the theorem of Pascal
Collin and the last author~\cite{CoRo}, extending the Jenkins-Serrin theory to
ideal domains of the hyperbolic plane $\H^2$.  They then used this to
construct an entire minimal graph over $\H^2$ in $\H^2\times\R$,
conformally the complex plane $\C$. Hence a harmonic diffeomorphism
from $\C$ onto $\H^2$.  This solved in the negative a conjecture of
Schoen and Yau: there is no harmonic diffeomorphism from $\C$ onto
$\H^2$.  We mention further generalizations of this
theorem~\cite{MaRoRo,LeRo, GaRo}.

To extend the Jenkins-Serrin theorem to higher topology Riemann
surfaces, a new idea was needed.  We solve this problem in this paper
with an idea introduced in the thesis of the first author~\cite{Maz1}: divergence
lines of sequences of minimal graphs.

\subsubsection*{Harmonic maps.}  Harmonic maps between surfaces have
long been used to study the Teichm\"uller space of a Riemann
surface. Sampson, Eells, Hartman were among the early
pioneers. They showed in \cite{EeSa} the existence of a harmonic map in
each homotopy class of maps from $M$ to $N$, when $N$ has non positive
sectional curvatures, and Hartman~\cite{Har} proved it is unique when
the curvature is strictly negative.

For closed hyperbolic surfaces of the same genus, Schoen and Yau~\cite{ScYa}
proved there is a unique harmonic diffeomorphism between them
homotopic to the identity.  Wolf, in~\cite{Wol2}, was able to
parametrize Teichm\"uller space by harmonic diffeomorphisms and
described the geometry of its closure (and other analytic properties)
in terms of the measured foliations of the Hopf differential of the
harmonic diffeomorphism.

Markovic~\cite{Mrk} extended this theory to non compact Riemann surfaces of
finite analytic type (conformally parabolic and finite topology).  He
studied the complex structures of such surfaces using quasi conformal
harmonic diffeomorphisms.

Harmonic diffeomorphisms from $\C$ into domains of $\H^2$ have been
analytically constructed by Au, Tam and Wan~\cite{AuTaWa}, Han, Tam,
Treibergs and Wan~\cite{HaTaTrWa} and Tam and Wan~\cite{TaWa}.  They showed
the image of the harmonic map is an ideal polygon of $\H^2$ with $m+2$
vertices, precisely when the Hopf differential is a polynomial of
degree $m$.

M. Wolf~\cite{Wol} realized that harmonic maps into $\H^2$ lead to minimal graphs
and multigraphs over domains in $\H^2$.  He made a construction of
such surfaces using harmonic maps to real trees and measured
foliations.  This gave many interesting examples.  He showed how the
measured foliations give information on the growth of the minimal
surfaces.

The main question of our study is to understand if the conformal types of two surfaces are
related if there is a harmonic diffeomorphism from one to the other. We finish this introduction by
stating a conjecture in that direction. There is no harmonic diffeomorphism from the disk
onto $\R^2$ with a complete parabolic metric.

The paper is structured as follows. In the second section, we recall some basic
definitions about conformal type, topology and geometry of surfaces.
Section~\ref{sec:nonexist} is devoted to the proof of a non existence result for harmonic
diffeomorphism. In Section~\ref{sec:mse}, we gather some results about the minimal surface
equation that are used in the next section to prove a Jenkins-Serrin type result. This
result is then used in Section~\ref{sec:construc} to construct a harmonic diffeomorphism
from a parabolic surface to a hyperbolic surface with infinite area.


\section{Preliminaries}

In this section we recall some basic facts about conformal type of surfaces, harmonic
maps, the geometry of hyperbolic surfaces.


\subsection{Conformal type}

We refer to \cite{Gri} for the notions introduced here. First we recall the following
definition.

\begin{defn}
Let $(M,g)$ be a complete Riemannian manifold with empty boundary. $(M,g)$ is called
parabolic if any bounded subharmonic function on $M$ is constant.

If $\partial M$ is compact, $(M,g)$ is called parabolic if any non negative bounded
subharmonic function which vanishes on $\partial M$ vanishes everywhere.
\end{defn}

If $(M,g)$ has no boundary and $K\subset M$ is a compact with smooth boundary, it is well
known that $M$ is parabolic if and only if $M\setminus \inter K$ is parabolic.

In dimension $2$, the parabolicity is a conformal property. For example, an annulus is
parabolic if and only if its conformal modulus is $+\infty$: we recall that any annular
domain with one connected compact boundary is conformal to $\{1\le z<c\}$,
$c\in(1,\infty]$, the conformal modulus of this annular domain in $\frac1{2\pi}\ln(c)$.


\subsection{Harmonic maps}

A harmonic map $\phi: (M_1,g_1)\to (M_2,g_2)$ between two Riemannian manifolds is a
critical point of the Dirichlet energy functional
$E(\phi)=\frac12\int_{M_1}|\phi^*g_2|_{g_1}^2d\sigma_{g_1}$ where $d\sigma_{g_1}$ is the volume
measure. If $M_1$ has dimension $2$, this
energy is conformally invariant, so being harmonic only depends on the conformal structure
of $M_1$. 

If $M_1$ and $M_2$ are surfaces, let us consider conformal parameters $z$ and $w$ on $M_1$
and $M_2$ and write their metrics as $g_1=\lambda^2(z)|dz|^2$ and $g_2=\sigma^2(w)|dw|^2$. Then a
map $\phi: M_1\to M_2$ can be written as a function $w=u(z)$. With these notations,
the map $\phi$ is harmonic if and only if $u$ satisfies the following partial differential
equation
$$
0=u_{z\bar z}+2\frac{\sigma_w(u)}{\sigma(u)}u_zu_{\bar z}
$$

Let $S$ be a surface and $M$ a Riemannian $3$-manifold. An isometric immersion $\phi :S\to
M$ is harmonic if and only if $\phi$ is minimal. In the case $M$ is a Riemannian
product $M=\Sigma\times\R$ and $\pi:M\to \Sigma$ denotes the vertical projection, 
$\phi$ minimal implies that the map $\pi\circ \phi: S\to \Sigma$ is harmonic.


\subsection{Hyperbolic surfaces}

In this paper we will look at orientable surfaces $\Sigma$ with finite topology and endowed with a
complete hyperbolic metric. Let us describe the geometry of the annular ends $E$ of $\Sigma$.

If $E$ has finite area then outside some compact $E$ is isometric to the quotient of a
horodisk $H$ by a parabolic translation leaving $H$ invariant. The quotient of a horodisk
$H'$ contained in $H$ is called a horoannulus of the end. Such an end can be compactified
by adding one point $p$ at infinity. These annular ends are parabolic. We call these ends hyperbolic cusp ends.

If $E$ has infinite area, the geometry is the following. First, we have two particular cases:
\begin{itemize}
\item $\Sigma$ could be $\H^2$ and $E$ is just the outside of  a compact subset of $\H^2$ or
\item $\Sigma$ could be the quotient of $\H^2$ by a parabolic translation and $E$ is just
the quotient of the outside of a horodisk left invariant by the parabolic translation.
\end{itemize}
If we are not in these two particular case, the picture is the following. Let $\gamma$ be
a geodesic in $\H^2$. If $c$ is an equidistant curve $c$ to $\gamma$, we denote $C_c$ the
non-convex component of $\H^2\setminus c$. Then there is an equidistant curve $c$ and a
hyperbolic translation $T$ leaving $\gamma$ invariant such that, outside some compact, $E$
is isometric to the quotient of $C_c$ by $T$. Thus, when we will consider such an end $E$,
we will see $E$ as the particular subdomain isometric to the quotient of $C_c$ by $T$, for
example $\partial E$ has constant
curvature. The ideal compactification of $\H^2$ passes to the quotient and gives a
compactification of $E$ by adding a circle. These annular ends are non-parabolic. In the
following, we will focus on this
general case. So any non parabolic end will be seen as the quotient of some $C_c$ by $T$
(the other cases can be treated similarly but are exactly the cases studied in \cite{CoRo,LeRo}).

Since each end can be compactified, the whole surface $\Sigma$ can be compactified by
$\barre\Sigma^\infty$. We will denote $\partial_\infty \Sigma=\barre\Sigma^\infty\setminus
\Sigma$. $\partial_\infty\Sigma$ is made of one point for each parabolic end and one
circle for each non-parabolic end. If $A$ is a subset of $\Sigma$, $\barre A^\infty$
denotes the closure of $A$ in $\barre\Sigma^\infty$ and $\partial_\infty A$ is $\barre
A^\infty\cap \partial_\infty\Sigma$ .


\section{A non-existence result}\label{sec:nonexist}

Our first result is a characterization "\`a la Huber" of parabolicity. For a Riemannian
metric $g$ on a surface we denote by $K_g$ the sectional curvature and $d\sigma_g$ the
area measure.

\begin{prop}\label{prop:huber}
Let $(S,g)$ be a complete Riemannian surface with a compact boundary. We denote $S_r=\{p\in
S|d_g(p,\partial S)<r\}$ . We assume that there is $C>0$ such that for any $r>0$
$$
\int_{S_r}K_gd\sigma_g\ge -C\ln(2+r).
$$
Then $(S,g)$ is parabolic.
\end{prop}

\begin{proof}
We are going to give an upper-bound for the growth of the area $|S_r|$ of $S_r$. Let
$\ell(r)$ denote the length of $\partial S_r\setminus \partial S$. It is known that $\ell$ is
differentiable almost everywhere and $\ell(b)-\ell(a)\le \int_a^b\ell'(u)du$ and
$\ell'(u)\le 2\pi\chi(S_u)-\boK(u)+\int_{\partial S}\kappa_g ds$ where
$\boK(u)=\int_{S_u}K_gd\sigma_g$ and $\int_{\partial S}\kappa_g ds$ is the integral of
the curvature of $\partial S$ computed with respect to the outward unit normal
(see \cite{Cas} and the references therein). We denote by $c$ this last integral. From the
coarea formula, we have:
\begin{align*}
|S_r|&=\int_0^r \ell(u)du\\
&\le r\ell(0)+\int_0^r\int_0^u\ell'(v)dvdu\\
&\le r\ell(0)+\int_0^r\ell'(v)\int_v^rdudv\\
&\le r\ell(0)+\int_0^r(r-v)\ell'(v)dv\\
&\le r\ell(0)+\int_0^r(r-v)(2\pi\chi(S_v)+c-\boK(v))dv\\
&\le r\ell(0)+\int_0^r(r-v)(2\pi+c+C\ln(2+v))dv\\
&\le r\ell(0)+(2\pi+c)\frac{r^2}2+Cr\int_0^r\ln(2+v)dv\\
&\le r\ell(0)+(2\pi+c)\frac{r^2}2+Cr^2\ln(2+r)
\end{align*}
So $\frac r{|S_r|}$ is not integrable at $+\infty$ which implies $(S,g)$ is parabolic (see \cite{Gri}).
\end{proof}

In \cite{Sch3}, Schoen proved that there is no harmonic diffeomorphism from the unit disk
onto a complete surface of non negative curvature. The following proposition is an
improvement of this result.

\begin{prop}
Let $S$ be an orientable surface with a compact boundary. Let $g_1$ and $g_2$ be two complete
Riemannian metrics on $S$. Assume that there is a constant $C\ge 0$ such that for any $r$
$$\int_{S_r^2}K_{g_2}^-d\sigma_{g_2}\le C\ln(2+r)$$
where $K_{g_2}^-=\max(0,-K_{g_2})$ and $S_r^2=\{p\in S|d_{g_2}(p,\partial S)<r\}$.

If there is a harmonic diffeomorphism $\phi :(S,g_1)\to (S,g_2)$ then $(S,g_1)$ is
parabolic.
\end{prop}

\begin{proof}
By changing the orientation on $(S,g_1)$, we can assume that $\phi$ preserves the orientation.
Let $z$ be a local conformal complex coordinate on $(S,g_1)$ and $w$ be a local conformal
complex coordinate on $(S,g_2)$. We denote $g_1=\lambda^2(z)|dz|^2$,
$g_2=\sigma^2(w)|dw|^2$ and $\phi$ by $w=u(z)$. The Jacobian of the map $u$ is then
$J(u)=\frac{\sigma^2(f)}{\lambda^2(z)}(|u_z|^2-|u_{\bar z}|^2)$ (see \cite{ScYa,Sch3}) since $u$ preserves the
orientation $J(u)>0$ and $|u_z|>|u_{\bar z}|$.

We then define on $S$ the metric $\hat g=\sigma^2(u)|u_z|^2|dz|^2$; this metric is
conformal to $g_1$ and does not depend on the choice of the complex coordinates $z$ and
$w$. If we compare $\hat g$ with the pull-back metric $\phi^* g_2$ we have
\begin{align*}
\phi^*g_2=\sigma^2(u)|u_z dz+u_{\bar z}d\bar z|^2
&\le\sigma^2(u)(|u_z| |dz|+|u_{\bar z}| |d\bar z|)^2\\
&\le \sigma^2(u)2(|u_z|^2+|u_{\bar z}|^2) |d z|^2\\
&\le \sigma^2(u)4|u_z|^2 |d z|^2\\
&\le 4\hat g
\end{align*}
Since $\phi^*g_2$ is complete, $\hat g$ is complete. Let 
\begin{gather*}
S_r^{\wedge}=\{p\in S\,|d_{\hat g}(p,\partial S)\le r\}\\
S_r^*=\{p\in S\,|d_{\phi^* g_2}(p,\partial S)\le r\}
\end{gather*}
Since $\phi^*g_2\le 4\hat g$, we have $S_r^\wedge\subset S_{2r}^*$.

The computation of the curvature of $\hat g$ (see \cite{ScYa,Sch3}) gives
$$
K_{\hat g}=\frac{-1}{2\sigma^2(u)|u_z|^2}\Delta \ln(\sigma(u)^2|u_z|^2)=
K_{g_2}(u)\frac{\sigma^2(u)(|u_z|^2-|u_{\bar z}|^2)}{\sigma^2(u)|u_z|^2}=K_{g_2}\hat J(\phi)
$$
where $\hat J(\phi)$ is the jacobian of $\phi:(S,\hat g)\to(S,g_2)$.
Thus
\begin{align*}
\int_{S_r^\wedge}K_{\hat g}^-da_{\hat g}
&=\int_{S_r^\wedge}K_{g_2}^-(\phi)\hat J(\phi) d\sigma_{\hat g}\\
&=\int_{\phi(S_r^\wedge)}K_{g_2}^-d\sigma_{g_2}\\
&=\int_{S_r^\wedge}K_{\phi^*g_2}^-d\sigma_{\phi^*g_2}\\
&\le\int_{S_{2r}^*}K_{\phi^*g_2}^-d\sigma_{\phi^*g_2}\le C\ln(2r+2)
\end{align*}
So, by Proposition~\ref{prop:huber}, $(S,\hat g)$ is parabolic. Since $\hat g$ is
conformal to $g_1$, $(S,g_1)$ is parabolic.
\end{proof}

As a consequence we have the following corollary

\begin{cor}\label{cor:cusp}
Let $g_1$ and $g_2$ be two complete metrics on an orientable finite topology surface $S$
with $g_2$  hyperbolic and of finite area. If there is a harmonic diffeomorphism $\phi :
(S,g_1)\to (S,g_2)$ then $(S,g_1)$ is parabolic.

There is no harmonic diffeomorphism from $A(r)=\{z\in\C\,|1\le |z|<r\}$ to a hyperbolic cusp end (we recall that a hyperbolic cusp end is parabolic).
\end{cor}


\section{Preliminaries on the minimal surface equation}\label{sec:mse}

Let $\Ome$ be an open subset inside a Riemanniann surface $\Sigma$ and $u$ be a function
on $\Ome$. In the following, we will use the following notations
\begin{itemize}
\item $G_u$ is the graph of $u$ in $\Sigma\times \R$,
\item $W_u$ is $\sqrt{1+\|\nabla u\|^2}$,
\item $X_u$ is the vectorfield $\frac{\nabla u}{W_u}$,
\item $N_u$ is $(X_u,-\frac1{W_u})$ the downward unit normal to $G_u$ and 
\item if $\gamma$ is a curve in $\Ome$, $F_u(\gamma)=\int_\gamma X_u\cdot\nu$ where $\nu$
is a unit normal to $\gamma$.
\end{itemize}
$F_u(\gamma)$ is called the flux across $\gamma$. Of course, $F_u(\gamma)$ depends on the
choice of $\nu$ but, in the following, $\gamma$ will be often a boundary component of
some open subset so $\nu$ will be always chosen as the outward pointing unit normal.

\subsection{The minimal surface equation}

The function $u$ solves the minimal surface equation if
\begin{equation*}\tag{MSE}\label{eq:mse}
0=\Div\left(\frac{\nabla u}{\sqrt{1+\|\nabla u\|^2}}\right)=\Div X_u
\end{equation*}
This is equivalent to say that $G_u$ is a minimal surface in $\Ome\times \R$.

We are going to study the Dirichlet problem for the~\eqref{eq:mse}. This problem has been
studied by many different authors. We refer to the works of Jenkins and Serrin \cite{JeSe}
for $\R^2$, Nelli and the last author \cite{NeRo}, the authors \cite{MaRoRo} for $\H^2$
and Pinheiro \cite{Pin} for the the general case. We will gather here some results whose
proof can be found in these papers.

The first result is a classical compactness result.

\begin{thm}\label{th:comp}
Let $(u_n)_n$ be a uniformly bounded sequence of solutions of the~\eqref{eq:mse} on an open
subset $\Ome$ of $\Sigma$. There is a subsequence of $(u_n)_n$ that converges to a solution
$u$ of the~\eqref{eq:mse}; the convergence is smooth on each compact subset of $\Ome$.
\end{thm}

We have introduced the notation $F_u(\gamma)$ for curves in $\Ome$, actually this notion
can be extended to subarcs of $\partial \Ome$ if $\Ome$ is smooth and $u$ solves
\eqref{eq:mse}. Indeed, $X_u$ is bounded and, in that case, has vanishing divergence.
Actually, we can often extend continuously the value of $X_u$ on $\partial\Ome$.

\begin{lem}\label{lem:bound}
Let $\Ome$ be an open subset in $\Sigma$ and $\gamma$ a geodesic arc contained in
$\partial\Ome$ and $u$ a solution of \eqref{eq:mse} in $\Ome$.
\begin{itemize}
\item If $u$ diverges to $+\infty$ (resp. $-\infty$) as one approaches $\gamma$, then
$X_u$ extends continuously on $\gamma$ with $X_u=\nu$ (resp. $X_u=-\nu$).
\item If $F_u(\gamma)=\ell(\gamma)$ (resp. $-\ell(\gamma)$), then $u$ diverges to
$+\infty$ (resp. $-\infty$) as one approaches $\gamma$.
\end{itemize}
\end{lem}

\begin{proof}
The first statement is almost contained in Lemma~2.5 in \cite{MaRoRo} where $u\to+\infty$
on $\gamma$ implies $F_u(\gamma)=\ell(\gamma)$ is proved. Actually, if $u\to +\infty$ on
$\gamma$, Proposition~\ref{prop:equi2} implies that $X_u$ is equicontinuous near $\gamma$
so it extends continuously to $\gamma$. The value of $X_u$ on $\gamma$ is then a
consequence of $F_u(\gamma)=\ell(\gamma)$.

The second statement is Lemma~3.6 in \cite{MaRoRo}.
\end{proof}

An other result is Lemma~2.7 in \cite{MaRoRo}.

\begin{lem}\label{lem:div}
Let $\Ome$ be an open subset in $\Sigma$ and $\gamma$ a geodesic arc contained in
$\partial\Ome$ and $(u_n)_n$ a sequence of solutions of \eqref{eq:mse} in $\Ome$ which
extend continuously to $\gamma$. If $(u_n)_n$ diverges to $+\infty$ (resp. $-\infty$) on
$\gamma$ while remaining bounded on compact subsets of $\Ome$ then $F_{u_n}(\gamma)\to
\ell(\gamma)$ (resp. $-\ell(\gamma)$).
\end{lem}


\subsection{Divergence lines}

One important tool for our study is to understand the limit of a sequence
of solutions of \eqref{eq:mse}. In this section, we present the notion of divergence line
that was introduced by the first author in \cite{Maz1} for $\R^2$ and the authors in
\cite{MaRoRo} for $\H^2$.

In the sequel, we consider a complete Riemannian surface $\Sigma$,  a sequence of open subsets $(\Ome_n)_n\subset
\Sigma$ and a sequence $(u_n)_n$ of solutions of \eqref{eq:mse}, $u_n$ being defined
on $\Ome_n$. First we define the limit open subset 
$$
\Ome=\bigcup_n\big(\textrm{interior}\Big(\bigcap_{k\ge n}\Ome_k\Big)\big)
$$

 Because of the equicontinuity result given by Proposition~\ref{prop:equi2}, we
can assume that the sequence $(X_{u_n})_n$ converges to some continuous vectorfield $X$
on $\Ome$ (the convergence is locally uniform). So we can define the convergence domain of
the sequence as the open subset
$$
\boB=\boB(X)=\{p\in \Ome\,|\,\|X\|(p)<1\}
$$
and the divergence set as $\boD=\boD(X)=\Ome\setminus \boB=\{p\in \Ome\,|\,\|X\|(p)=1\}$.

\begin{prop}\label{prop:divline}
Let $p$ be a point in $\Ome$.
\begin{itemize}
\item If $p\in \boB$, let $D$ be the connected component of $\boB$ containing $p$. Then
$u_n-u_n(p)$ converges on $D$ to a solution of \eqref{eq:mse} (the convergence is locally
$C^k$ for any $k$).
\item If $p\in \boD$, let $\gamma$ be the geodesic in $\Ome$ passing through $p$ and
orthogonal to $X(p)$. Then $\gamma\subset \boD$ and, for any $q\in \gamma$, $X(q)$ is the
unit normal to $\gamma$.
\end{itemize}
\end{prop}

\begin{proof}
Let us first assume that $p\in \boB$. On $D$, since $\|X\|<1$, we have $\nabla u_n\to
\frac X{\sqrt{1-\|X\|^2}}$. Thus
$u_n-u_n(p)$ converges locally in $C^1$ to a function $v$. Besides $u_n-u_n(p)$ being a
solution of \eqref{eq:mse}, Theorem~\ref{th:comp} implies that the convergence is locally in
$C^k$ for any $k$ and $v$ is a solution of \eqref{eq:mse}.

Let us now assume $p\in \boD$. Since $\|X\|(p)=1$, we have $N_{u_n}(p)\to X(p)$. There is
$\delta>0$ such that $G_{u_n}$ contains the geodesic disk of radius $\delta$ around
$(p,u_n)$ (take $\delta$ such that $2\delta\le d(p,\partial\Ome)$). Moreover, by curvature
estimates in \cite{RoSoTo}, the second fundamental form
of these graphs is uniformly bounded. As a consequence, after a vertical translation by
$-u_n(p)\partial_t$, this sequence of geodesic disks converges to a minimal disk $S$ of radius
$\delta$ which is orthogonal to $X(p)$ at $(p,0)$. Let $\theta=\la N,\partial_t\ra$ along
$S$, where $N$ is the unit-normal to $S$. Since $S$ is a limit of graphs $\theta \le 0$ and $\theta(p,0)=0$. Moreover $\theta$ is in the
kernel of the Jacobi operator: $0=\Delta_S\theta+(\Ric(N,N)+\|A\|^2)\theta$. So by the maximum
principle, $\theta=0$ along $S$. This implies that $S$ is contained in some $\gamma\times\R$
where $\gamma$ is a geodesic of $S$. Since $X$ is normal to $S$ at $(p,0)$, $\gamma$ is
normal to $X$ at $p$ as well. So $S$ is a geodesic disk of radius $\delta$ in $\gamma\times\R$. This
implies that, along the geodesic segment in $\gamma$ of length $2\delta$ and midpoint $p$,
$X$ is the unit normal to $\gamma$. Let $\tilde\gamma$ denote the connected component of
$\gamma\cap \Ome$ containing $p$. It is now clear that the subset of points $q$ in
$\tilde\gamma$ where $X(q)$ is the unit normal to $\tilde\gamma$ is open and closed in
$\tilde\gamma$ so it is the whole $\tilde\gamma$ and the second statement of the
proposition is proved.
\end{proof}

The above proposition tells that on each connected component of $\boB$ the sequence
$(u_n)$ converges up to a vertical translation. We also see that $\boD$ is made of
geodesics of $\Ome$ that we will call divergence lines of $X$. We
notice that since the unit normal to these geodesics is given by $X$, they are
embedded geodesics (perhaps periodic). The next lemma is important in order to describe the possible
divergence lines.

\begin{lem}
Let $\gamma$ be a divergence line, then it is a proper geodesic in $\Ome$.
\end{lem}

\begin{proof}
Assume that $\gamma$ is not a proper geodesic. So we can consider a arc-length
parametrization of $\gamma :\R_+\to \Ome$ and a sequence $(s_i)_i$ in $\R_+$ with
$s_{i+1}>s_i+1$ and $\gamma(s_i)\to p\in \Ome$. Let $r>0$ be such that the geodesic disk
$D(p,r)\subset \Sigma$ is convex, is included in $\Ome$, has area at most $r$ and the
length of $\partial D(p,r)$ is at most $7r$. By changing the sequence $(s_i)_i$, we assume
that $\gamma(s_i)\in D(p,r/2)$. This implies that the geodesic segment in $\gamma$ of length $r$ and
midpoint $\gamma(s_i)$ is included in $D(p,r)$. We notice that $D(p,r)\subset
\textrm{interior}\Big(\bigcap_{k\ge n_0}\Ome_k\Big)$ for some $n_0$.

Changing $u_n$ into $u_n-u_n(\gamma(0))$, we assume $u_n(\gamma(0))=0$. We are going to
estimate the area of $G_{u_n}$ inside $D(p,r)\times(-1,1)$ in two ways.

First let us compute an upper-bound. Let us define
$B_n=\{(q,t)\in D(p,r)\times(-1,1)\,|t<u_n(p)\}$. We have
$G_{u_n}\cap(D(p,r)\times(-1,1))\subset \partial B_n$ and $\partial B_n\setminus G_{u_n}$
has area at most the one of $\partial (D(p,r)\times(-1,1))$: $2r+2\times 7r=16r$. Since
$G_{u_n}$ is area minimizing in $D(p,r)\times \R$ we obtain the area of $G_{u_n}\cap
(D(p,r)\times(-1,1))$ is at most $16r$. 

Let us now compute a lower-bound. Let $U\Subset\Ome$ be an open subset containing
$\gamma[0,s_9]$ (we also have
$U\subset \textrm{interior}\Big(\bigcap_{k\ge n_0}\Ome_k\Big)$ for some $n_0$). By curvature
estimates~\cite{RoSoTo}, the curvature of the graphs $G_{u_n}$ over $U$ is uniformly
bounded. Let $S_n$ be the connected component of $G_{u_n}\cap(U\times(-1,1))$ containing
$(\gamma(0),0)$. As in the proof of Proposition~\ref{prop:divline}, the sequence $(S_n)_n$
converges to $\tilde\gamma\times(-1,1)$ where $\tilde\gamma$ is the connected component of
$\gamma\cap U$ containing $\gamma(0)$. $\tilde\gamma$ contains the geodesic segments of length $r$ and
midpoint $\gamma(s_i)$ ($0\le i\le 8$). So the area of the limit surface inside
$D(p,r)\times(-1,1)$ is at least $18r$. This implies that the area of $G_{u_n}\cap
(D(p,r)\times(-1,1))$ is at least $17r$ for $n$ large. We thus have a contradiction.
\end{proof}

Now we are interested in arguments that prevent some geodesics from being divergence lines. One
of these tools is the following result.

\begin{lem}\label{lem:nodiv1}
Let $\gamma$ be a divergence line and $p\in\gamma$ be a point. Let $D^+\subset \Ome$ be a
halfdisk centered at $p$ and contained strictly on one side of $\gamma$ and $\nu$ be the outward
pointing unit normal along $\partial D^+$. We assume that $D^+\subset\boB$ and consider $q\in
D^+$. Then if $X=\nu$ (resp. $X=-\nu$) along $\gamma$ then $\lim u_n(p)-u_n(q)=+\infty$
(resp. $-\infty$).
\end{lem}

\begin{proof}
Since $D_+\subset \boB$, we can assume that $u_n-u_n(q)\to v$, $v$ being  a solution to \eqref{eq:mse}.
We have $X_v=X$ so $X_v=\nu$ along $\gamma$. By Lemma \ref{lem:bound}, this implies that
$v$ takes the value $+\infty$ on $\gamma$. Let $M$ be positive. By the continuity of $X$,
there is a point $q'\in D^+$ close to $p$ such that $v(q')-v(q)\ge M$ and a curve $c$ in
$\barre D^+$ from $q'$ to $p$ such that $c'\cdot X>0$. So $c'\cdot X_{u_n}>0$ for large $n$ and 
$u_n(p)-u_n(q)\ge u_n(q')-u_n(q)$. Then $\liminf u_n(p)-u_n(q)\ge \liminf
u_n(q')-u_n(q)\ge M$ which gives the result.
\end{proof}

A consequence of the above lemma is the following. Let $c_1$ and $c_2$ be two connected
components of $\boB$ with a common divergence line in their boundary and $X$ pointing
into $c_2$ along it. Let $p_i\in c_i$ be two points, then $u_n(p_2)-u_n(p_1)\to +\infty$. We state the
next result only in the case where the surface is hyperbolic.

\begin{lem}\label{lem:nodiv2}
Let $(\gamma_n)_n$ be a sequence of geodesic arcs of length $2\delta>0$ with midpoints
$p_n$. Let $D_n^+$ be a geodesic half-disk with diameter $\gamma_n$, of radius $\delta$ and
strictly on one side of $\gamma_n$. We assume that $(D_n^+)_n$ converges to
$D^+$ a geodesic half-disk of center $p$, radius $\delta$ and on one side of a geodesic arc
$\gamma$. We assume that $(u_n)_n$ is a sequence of solutions of \eqref{eq:mse} on $D_n^+$
and $X_{u_n}\to X$ on $D^+$. Moreover we assume one of the following possibilities
\begin{itemize}
\item either $u_n$ takes the value $+\infty$ (resp. $-\infty$) along $\gamma_n$,
\item or the metric is hyperbolic on $D_n^+$ and $u_n$ is constant along $\gamma_n$.
\end{itemize}
In both cases, $p$ is not the end point of some divergence line of $X$. Moreover, in the
first case $X$ takes the value $\nu$ (resp. $-\nu$) along $\gamma$.
\end{lem}

\begin{proof}
By Proposition~\ref{prop:equi2}, in the first case, the sequence $X_{u_n}$ is uniformly
equicontinuous on the halfdisk of radius $\delta/2$. Since $u_n$ takes the value
$+\infty$, $X_{u_n}=\nu$ along $\gamma_n$. As a consequence, $X$ extends continuously to
$\gamma$ by the value $\nu$. This implies that $p$ is not the end point of some divergence
line of $X$.

In the second case, let us see the halfdisk $D_n^+$ as a halfdisk in $\H^2$. Since $u_n$
is constant along $\gamma_n$ we can extend the definition of $u_n$ to the whole geodesic
disk by symmetry. As above this implies that $X_{u_n}$ is uniformly equicontinuous on the
the halfdisk of radius $\delta/2$ and $X_{u_n}$ is orthogonal to $\gamma_n$ along it. So
$X$ extends continuously to $\gamma$ and is orthogonal to it. This prevents $p$ from being an
endpoint of some divergence line of $X$.
\end{proof}


\section{A Jenkins-Serrin type result}

In this section, we are interested in solving the Dirichlet problem for \eqref{eq:mse} on
some particular domains $\Ome$ of a complete hyperbolic surface $\Sigma$. So let us fix a
a complete hyperbolic surface $\Sigma$ with at least one non-parabolic end.


\subsection{Ideal domains and Jenkins-Serrin conditions}

We first define the notion of polygonal domains.

\begin{defn}\label{def:poldom}
A polygonal domain in $\Sigma$ is a connected open subset $\Ome$ such that
$\partial_\infty\Ome$ is made of a finite number of points and $\partial\Ome$ is made of a
finite number of geodesic arcs.
\end{defn}

If $\Ome$ is a polygonal domain, the geodesic arcs in the boundary are called the edges of
$\Ome$, the end points of these edges and points in $\partial_\infty \Ome$ are called the
vertices of $\Ome$. The
natural orientation of $\partial\Ome$ allows us to say that the edge $\gamma_2$ is the
successor of the edge $\gamma_1$ if $\gamma_2$ comes just after $\gamma_1$ when traveling along
$\partial\Ome$.

Let us remark that if $\gamma\subset \partial\Ome$ is a geodesic arc, it could be possible
that $\Ome$ is on both sides of $\gamma$. This implies that $\gamma$ is part of two edges
of $\Ome$ and, in the following, this arc has to be counted twice.

Among polygonal domains, we consider particular ones. Let $E_1,\dots, E_q$ be the
non-parabolic ends of $\Sigma$ ($q\ge 1$) and $p_{q+1},\dots, p_{q+n}$ the end-points of the
cusp ends. 

\begin{defn}
An ideal domain $\Ome$ in $\Sigma$ is a polygonal domain such that 
\begin{itemize}
\item $\partial_\infty\Ome =\{p_{q+1},\dots,p_{p+n}\}\bigcup \cup_{i=1}^q\{p_1^i,\dots,p_{2n_i}^i\}$
where $\{p_j^i\}_{1\le j\le 2n_i}$ are an even number of points in $\partial_\infty E_i$
cyclically ordered,
\item the edges of $\Ome$ are geodesic lines $\gamma_j^i$, $1\le i\le q$ and $1\le
j\le 2n_i$, where the end points of $\gamma_j^i$ are $p_j^i$ and $p_{j+1}^i$ (with
$p_{2n_i+1}^i=p_1^i$) and
\item the edge $\gamma_j^i$ is included in $E_i$.
\end{itemize}
(see Figure~\ref{fig:ideal})
\end{defn}

\begin{figure}[h]
\begin{center}
\resizebox{0.9\linewidth}{!}{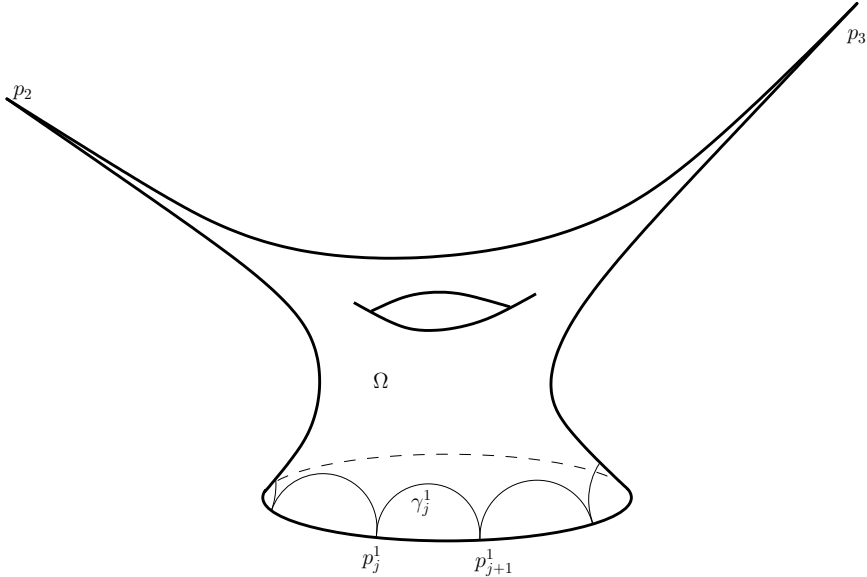}
\caption{An ideal domain $\Ome$ in a hyperbolic surface $\Sigma$ \label{fig:ideal}}
\end{center}
\end{figure}

As a consequence, the boundary of an ideal domain $\Ome$ is made of an even number of
edges. In the following, each edge will be labeled "$a$" or "$b$" with the convention : two
successive edges have different labels. Since the number of end points on $\partial_\infty
E_i$ is even such a labeling is possible. 

\begin{defn}
Let $\Ome$ be an ideal domain in $\Sigma$. An inscribed polygonal domain in $\Ome$ is a
polygonal domain contained in $\Ome$ whose vertices are among the ones of $\Ome$.
\end{defn}

Let us notice that the edges of such an inscribed polygonal domain are either closed
geodesics or complete geodesics. Besides, $\Ome$ is itself a polygonal domain inscribed
in
$\Ome$. Jenkins-Serrin conditions take into account the "lengths" of boundary components of
inscribed polygonal domains. So let us explain how these conditions are defined.

Let $\Ome$ be an ideal domain in $\Sigma$ and consider the vertices $p_j^i$ as in the
definition. For each $i,j$, let $(H_j^i(t))_{t\ge 0}$ be a decreasing family of horodisks
centered at $p_j^i$ such that $d(\partial H_j^i(0),H_j^i(t))=t$. In the cusp end with end
point $p_i$, we also consider a decreasing family of horo-annuli $(H_i(t))_{t\ge 0}$ such that
$d(\partial H_i(0),H_i(t))=t$. Moreover we assume that the horodisks $H_j^i(0)$ and the
horo-annuli $H_i(0)$ are disjoint. If $\mathbf
t=(t_1^1,\dots,t_{2n_1}^1,t_1^2,\dots,t_{2n_q}^q,t_{q+1},\dots,t_{q+n})\in\R_+^{N_{\Ome}}$
where $N_\Ome=n+\sum_{i=1}^q2n_i$, we define
$$
H(\mathbf t)=(\cup_{i=1}^q\cup_{j=1}^{2n_i} H_j^i(t_j^i))\bigcup (\cup_{k=1}^nH_{q+k}(t_{q+k}))
$$
This is the union of disjoint horodisks and horo-annuli.

Let us fix a $a/b$ labeling on $\partial \Ome$ and choose $\boP$ an inscribed polygonal
domain in $\Ome$. The edges of $\boP$ can be gathered in three classes: the
ones which are edges of $\Ome$ labeled $a$ (we denote by $A_\boP$
the union of these geodesic lines), the ones which  are edges of
$\Ome$ labeled $b$ (let $B_\boP$ be their union), the other ones (let $C_\boP$ be their
union).

For $\bft\in \R_+^{N_\Ome}$, we define $A_\boP(\bft)=A_\boP\setminus H(\bft)$,
$B_\boP(\bft)=B_\boP\setminus H(\bft)$ and $C_\boP(\bft)=C_\boP\setminus H(\bft)$. We also
denote $\alpha(\bft)=\ell(A_\boP(\bft))$, $\beta(\bft)=\ell(B_\boP(\bft))$ and
$\gamma(\bft)=\ell(A_\boP(\bft)\cup B_\boP(\bft)\cup C_\boP(\bft))$ where $\ell$ denotes
the length of a curve.

On $\R_+^{N_\Ome}$, we define a partial order by $\bft'\ge \bft$ if $\bft'-\bft$ has only
non negative components. 

Let $\gamma$ be an edge of $\boP$. We notice that if all the components of
$\bft$ are sufficiently large then $\gamma$ only intersects the horodisks or horo-annuli in
$H(\bft)$ that are centered at the end points of $\gamma$. We assume this is true in the
following. Let us understand how the three above quantities evolve when coordinates in $\bft$
increase. The edges with the vertex $p_k$ as end-point are included in $C_\boP$. So increasing
$t_k$ by $t$, leave $\alpha(\bft)$ and $\beta(\bft)$ unchanged and increase $\gamma(\bft)$ by
at least $2t$ (there are at least two edges ending at $p_k$: it could be one geodesic line
counted twice). If $p_j^i\in\partial_\infty\boP$, when $t_j^i$ increases to $t_j^i+t$,
either $\alpha(\bft)$ (resp.
$\beta(\bft)$) increases by $t$ or stays unchanged, depending on whether an edge in $A_\boP$
(resp. $B_\boP$) ends at $p_i^j$; in any case $\gamma$ increases by at least $2t$. 

As a consequence, $\gamma(\bft)-2\alpha(\bft)$ and $\gamma(\bft)-2\beta(\bft)$ is non decreasing with $\bft$. So
the Jenkins-Serrin conditions $\gamma-2\alpha>0$ and $\gamma-2\beta>0$ are well defined for any inscribed
polygon $\boP$ and means that $\gamma(\bft)-2\alpha(\bft)>0$ and
$\gamma(\bft)-2\beta(\bft)>0$ for sufficiently large $\bft$.

If $\boP=\Ome$, the same argument proves that the condition $\alpha-\beta=0$ is well
defined since the value $\alpha(\bft)-\beta(\bft)$ does not depend on $\bft$ for large
$\bft$.

\begin{remarq}\label{rm:alternate}
The above analysis has the following consequence. If $\gamma_1$ and $\gamma_2$
are two successive edges of $\boP$ and none of them is included in $A_\boP$, we see that
$\gamma(\bft)-2\alpha(\bft)\to +\infty$ as the components of $\bft$ go to $+\infty$. So
the condition $\gamma-2\alpha>0$ is always satisfied for such an inscribed polygonal
domain. So, this condition can be studied only for inscribed polygonal domains $\boP$ such
that $a$ edges alternate along $\partial\boP$. For the $\gamma-2\beta$ condition, we can
focus on inscribed polygonal domains $\boP$ such that $b$ edges alternate along
$\partial\boP$
\end{remarq}

\begin{lem}\label{lem:finite}
let $\Ome$ be an ideal domain in $\Sigma$. There is a number $L(\Ome)$ depending only on $\Ome$ such the following is true. Let $\{\gamma_i\}_{1\le i \le L}$ be a set of
disjoint proper geodesics in $\Ome$ which are either closed or with end-points among the
vertices of $\Ome$. Then $L\le L(\Ome)$.
\end{lem}

\begin{proof}
First, we look at closed geodesics. If such a geodesic $\gamma$ bounds a topological
disk $D$ then the Gauss-Bonnet formula gives $-|D|=2\pi$ so none of these geodesics are
homotopically trivial. If two of them bound a topological annulus $A$, the Gauss-Bonnet formula
gives $-|A|=0$ so any two closed geodesics are not homotopic. So there is a constant
$\kappa_\Sigma$ depending only on the topology of $\Sigma$ such that the number of closed
geodesics in $\{\gamma_i\}_{1\le i \le L}$ is less than $\kappa_\Sigma$.

So now we remove the closed geodesics from $\{\gamma_i\}_{1\le i \le L}$ and we consider
the connected components $\boP$ of the complement of these geodesics. Applying Gauss-Bonnet
formula to $\Ome$, we obtain $-|\Ome|+(N_\Ome-n)\pi=2\pi\chi(\Ome)$ (let us recall
that $N_\Ome-n$ is the number of vertices of $\Ome$ on non parabolic ends). So
$|\Ome|=(N_\Ome-n)\pi-2\pi\chi(\Ome)$. From the Gauss-Bonnet formula we also obtain
$-|\boP|+(n_\boP+k_\boP)\pi=2\pi\chi(\boP)$ where $\boP$ is a connected component of
$\Ome\setminus(\bigcup_{1\le i\le L}\,\gamma_i)$, $n_\boP$ is the number of edges of $\boP$
among the edges of $\Ome$ and $k_\boP$ the number of edges among $\{\gamma_i\}_{1\le i \le
L}$. As a consequence, the area $|\boP|$ is an integer multiple of $\pi$ and the
number of such $\boP$ is less than
$|\Ome|/\pi=N_\Ome-n-2\chi(\Ome)$. Besides we have
$$
k_\boP \pi-2\pi\le(n_\boP+k_\boP)\pi-2\pi\chi(\boP)=|\boP|\le |\Ome|;
$$
so $k_\boP\le 2+|\Ome|/\pi=2+N_\Ome-n-2\chi(\Ome)$. So summing over all $\boP$ and using the
estimate of the number of closed geodesics we have
$$
L\le \kappa_\Sigma+(1+\frac{N_\Ome-n}2-\chi(\Ome))(N_\Ome-n-2\chi(\Ome)).
$$
\end{proof}


\subsection{A Jenkins-Serrin theorem}

Let $\Ome$ be an ideal domain with a $a/b$ labeling of $\partial\Ome$. We are interested
in solving the following Dirichlet problem on $\Ome$ that we call the
Jenkins-Serrin-Dirichlet problem:
\begin{equation}\label{eq:jsd}
\begin{cases}
&\Div\left(X_u\right)=0 \text{ on } \Ome\\
&u=+\infty \text{ on } A_\Ome\\
&u=-\infty \text{ on } B_\Ome.
\end{cases}
\end{equation}

\begin{thm}\label{th:js}
Let $\Ome$ be an ideal domain with a $a/b$ labeling of $\partial \Ome$. The
Jenkins-Serrin-Dirichlet problem has a solution if and only if $\alpha-\beta=0$ for
$\boP=\Ome$, and 
$$\gamma-2\alpha>0 \quad\text{and}\quad \gamma-2\beta>0$$
for all other inscribed polygonal domains $\boP$. Moreover if the solution exists, it is
unique up to an additive constant.
\end{thm}

We separate the proof of the above theorem in three parts.


\subsubsection{The conditions are necessary}\label{sec:necessar}

Let $u$ be a solution and consider an inscribed polygonal domain $\boP$ and $\bft\in
\R_+^{N_\Ome}$ with large coordinates. The boundary of $\boP\setminus H(\bft)$ is made
of $A_\boP(\bft)$, $B_\boP(\bft)$, $C_\boP(\bft)$ and arcs with curvature $1$ contained in
$\partial H(\bft)$, we denote by $\Gamma_\bft$ the union of these arcs. We notice that
$\ell(\Gamma_\bft)$ goes to $0$ as $\bft\to\infty$. Since $u$ solves \eqref{eq:mse},
Lemma~\ref{lem:bound} gives
\begin{equation}\label{eq:flux}
\begin{split}
0=F_u(\partial(\boP\setminus H(\bft)))&=F_u(A_\boP(\bft))+F_u(B_\boP(\bft))+ F_u(C_\boP(\bft))+
F_u(\Gamma_\bft)\\
&=\alpha(\bft)-\beta(\bft)+ F_u(C_\boP(\bft))+ F_u(\Gamma_\bft)
\end{split}
\end{equation}
Since $\|X_u\|<1$ along $C_\boP(\bft)$ and $\Gamma_\bft$, we have $|F_u(\Gamma_\bft)|\le
\ell(\Gamma_\bft)\xrightarrow[\bft \to \infty]{}0 $ and, if $\boP\neq \Ome$,
$C_\boP(\bft)$ is nonempty then
$|F_u(C_\boP(\bft))|<\ell(C_\boP(\bft))=\gamma(\bft)-\alpha(\bft)-\beta(\bft) $. Moreover
the difference $\gamma(\bft)-\alpha(\bft)-\beta(\bft)- |F_u(C_\boP(\bft))|$ is non
decreasing with $\bft$. So, if $\boP\neq \Ome$, there is $c>0$ such that for $\bft$ large
$\gamma(\bft)-\alpha(\bft)-\beta(\bft)- |F_u(C_\boP(\bft))|\ge
c$. Using this in \eqref{eq:flux}, we obtain
$$
0\le \alpha(\bft)-\beta(\bft)+(\gamma(\bft)-\alpha(\bft)-\beta(\bft)-c)+\ell(\Gamma_\bft)
$$
which implies $\gamma(\bft)-2\beta(\bft)\ge c/2>0$ for $\bft$ large enough. So
$\gamma-2\beta>0$ on $\boP$. Similar computations give $\gamma-2\alpha>0$ on $\boP$. If
$\boP=\Ome$, taking the limit in \eqref{eq:flux} gives $\lim_{\bft\to
\infty}\alpha(\bft)-\beta(\bft)=0$, so $\alpha-\beta=0$ for $\boP=\Ome$.

\begin{remarq}\label{rem:uniform}
If $\gamma$ is a subarc of $C_\boP$ and $\|X_u\|\le 1-\delta$ ($\delta>0$), the constant
$c$ appearing in the above proof can be taken equal to $\delta \ell(\gamma)$. 

A second remark is that if $\boP$ is a polygonal domain as in Definition~\ref{def:poldom}
which is contained in $\Ome$. We can also define $A_\boP$, $B_\boP$ and $C_\boP$ and look
at the $\gamma-2\alpha>0$ and $\gamma-2\beta>0$ conditions. The arguments above also tell us
that these conditions are satisfied for such polygonal domains.
\end{remarq}


\subsubsection{The existence part}

The first step of the existence part of Theorem~\ref{th:js} proof is given by the following result.

\begin{lem}\label{lem:perron}
Let $\Ome$ be an ideal domain with a $a/b$ labeling of $\partial \Ome$. For any $n$, there
is a solution to the following Dirichlet problem in $\Ome$
\begin{equation}\label{eq:perron}
\begin{cases}
&\Div\left(X_u\right)=0 \text{ on } \Ome\\
&u=n \text{ on } A_\Ome\\
&u=-n \text{ on } B_\Ome.
\end{cases}
\end{equation}
\end{lem}

\begin{proof}
We apply the Perron method to solve this Dirichlet problem (see Theorem 2.12 in
\cite{GiTr} for harmonic functions). Let us recall its framework. A continuous
function $w$ on $\Ome$ is called a subsolution of \eqref{eq:mse} if, for any
bounded open subset $U\subset\subset\Ome$ with smooth boundary and any solution $v$ to
\eqref{eq:mse} on $U$, $w\le v$ on $\partial U$ implies $w\le v$ on $U$.

A continuous function $w$ on $\barre\Ome$ is called a subsolution to \eqref{eq:perron} if
it is a subsolution to \eqref{eq:mse} and $w\le n$ on $A_\Ome$ and $w\le -n$ on $B_\Ome$.
 Let $\boS$ be the set of all subsolutions to \eqref{eq:perron}. We notice that 
 $w\equiv -n\in \boS$ and any $w\in \boS$ satisfies $w\le n$. The Perron method asserts
 that the function $u$ defined on $\Ome$ by $u(p)=\sup_{w\in \boS}w(p)$ solves
 \eqref{eq:mse}.

The fact that $u$ satisfies the boundary data of \eqref{eq:perron} comes from the
existence of barriers along the boundary. They can be constructed as follows. Take a point
$p$ in $\partial\Ome$ and consider $D^+$, a geodesic half-disk contained in $\Ome$ and
centered at $p$. There exists a solution $v$ of \eqref{eq:mse} on $D^+$ with boundary data
$0$ on $\partial D^+\cap\partial\Ome$ and $2n$ on $\partial D^+\cap \Ome$. Then, if $p\in
A_\Ome$,  $n-v\le u\le n$ on $D^+$ since $n-v$ is a subsolution and then $u(p)=n$. If $p\in B_\Ome$, we have
$-n\le u\le v-n$ and $u(p)=-n$.
\end{proof}

Let $(u_n)_n$ be the sequence of solutions given by Lemma~\ref{lem:perron}. Let us prove
that, up to a subsequence, the sequence converges to a solution of Problem~\eqref{eq:jsd}. We
assume that $X_{u_n}\to X$ and the question is to understand the possible divergence
lines of $X$. Each of them are proper geodesics in $\Ome$ and, since $u_n$ is locally
constant along $\partial \Ome$, their end points must be among the vertices of $\Ome$ (Lemma~\ref{lem:nodiv2}). So
divergence lines are either closed geodesics or geodesic lines joining two vertices of
$\Ome$. 

First let us assume that we have at least one divergence line (the convergence
domain $\boB(X)$ is not the whole $\Ome$). There
are at most a finite number of divergence lines (Lemma~\ref{lem:finite}). Thus $\boB(X)$ has a finite
number of connected components. Let us define an oriented graph $G$ in the following way.
The vertices of $G$ are the connected components of $\boB(X)$. A divergence
line $\gamma$ lies in the boundary of two connected components $c_1$ and $c_2$ of
$\boB(X)$ (may be $c_1=c_2$) and, along $\gamma$, $X$ points into one of these
connected components, say $c_2$. We then define an arrow (or oriented edge) $e_\gamma$
from $c_1$ to $c_2$. $G$ is then a finite oriented graph.

\begin{lem}
$G$ has no oriented cycle.
\end{lem}

\begin{proof}
Assume $e_{\gamma_1}\cdots e_{\gamma_k}$ is an oriented cycle in $G$. Let $c_i$ be the
initial point of $e_{\gamma_i}$; as a cycle, the edge $e_{\gamma_k}$ has endpoint $c_1$.
Let $q_i$ be a point in $c_i$. By Lemma~\ref{lem:nodiv1}, we have
$$
0=(u_n(q_1)-u_n(q_k))+\sum_{i=1}^{k-1}(u_n(q_{i+1})-u_n(q_i))\xrightarrow[n\to\infty]{}+\infty
$$
which gives a contradiction.
\end{proof}

So $G$ has a vertex $c$ where all adjacent arrows arrive. The component $c$ is an
inscribed polygonal domain $\boP$ in $\Ome$. Let $q\in \boP$ and define $w$ on $c$ as the
limit of $u_n-u_n(q)$. Since $\boB(X)\neq \Ome$ and $g$ has no oriented cycle, there is an
other vertex $c'$ in $G$ which is joined to $c$ by some edge $e_\gamma$. As a consequence
if $q'\in c'$, we have $u_n(q)-u_n(q')\to +\infty$ (Lemma~\ref{lem:nodiv1}).
Since $u_n\ge -n$, this implies
$u_n(q)+n\to +\infty$. We have then proved that $w=-\infty$ on $B_\boP$ (the
edges of $\boP$ among the $b$-edges of $\Ome$) and $X=-\nu$ along $B_\boP$
(Lemmas~\ref{lem:bound} and \ref{lem:div}).

Let $\bft\in\R^{N_\Ome}$ be large. As above, the boundary of $\boP\setminus H(\bft)$ splits
into $A_\boP(\bft)$, $B_\boP(\bft)$, $C_\boP(\bft)$ and $\Gamma_\bft$. So we can compute
$$
0=F_{u_n}(\partial (\boP\setminus H(\bft)))= F_{u_n}(A_\boP(\bft))+
F_{u_n}(B_\boP(\bft))+ F_{u_n}(C_\boP(\bft))+ F_{u_n}(\Gamma_\bft)
$$
Taking the limit $n\to \infty$ and using $X=-\nu$ along $B_\boP(\bft)$ and
$C_\boP(\bft)$ and $\|X\|\le 1$ along $A_\boP(\bft)$ and $\Gamma_\bft$, we obtain
$$
\beta(\bft)+(\gamma(\bft)-\alpha(\bft)-\beta(\bft))\le \alpha(\bft)+\ell(\Gamma_\bft)
$$
So making $\bft\to\infty$, we get $\liminf_{\bft\to\infty}\gamma(\bft)-2\alpha(\bft)\le 0$
which is impossible since we assume $\gamma-2\alpha>0$ for $\boP$. As a consequence, we
have proved that there is no divergence line and $\boB(X)=\Ome$. Let us notice that we
can do the same argument with a vertex $c$ where all adjacent arrows leave and obtain a
contradiction with the $\gamma-2\beta>0$ property.

Taking $p\in\Ome$, we define $w$ the limit of $u_n-u_n(p)$ on $\Ome$. The function $w$ has the
right boundary values. Indeed, because of the values of $u_n$ on $\partial\Ome$, we can be sure that
either $w=+\infty$ on $A_\Ome$ or $w=-\infty$ on $B_\Ome$ (Lemmas~\ref{lem:bound} and
\ref{lem:div}). We assume $w=+\infty$ on $A_\Ome$ (the other case is similar).

For $\bft$ large, we have
$$
0=F_w(\partial (\Ome\setminus H(\bft)))=F_w(A_\Ome(\bft))+F_w(B_\Ome(\bft))+
F_w(\Gamma_\bft)
$$
Let us fix some $\bft_0$ and assume that $F_w(B_\Ome(\bft_0))\ge -\beta(\bft_0)+c$ for some positive
$c$. Then for any $\bft \ge \bft_0$, using $X_w=\nu$ on $A_\Ome$ and $\|X_w\|\le 1$ on
$B_\Ome$, the above equality gives $\alpha(\bft)-\beta(\bft)\le -c+
\ell(\Gamma_\bft)$. So $\alpha(\bft)-\beta(\bft)\le -c/2<0$ for $\bft$ large.
This gives a contradiction with $\alpha=\beta$ for $\Ome$. So
$F_w(B_\Ome(\bft))= -\beta(\bft)$ for large $\bft$ and $w=-\infty$ on $B_\Ome$ (Lemma~\ref{lem:bound}).


\subsubsection{The uniqueness part}

Let $u$ and $v$ be two solutions of \eqref{eq:jsd} and assume that $u-v$ is not a constant. Let $t$ be a
regular value of $u-v$ in the range of $u-v$ and define $D=\{u-v>t\}$, we notice that
along $\partial D\cap \Ome$ which is non-empty, $X_u-X_v$ points inside $D$. Let
$\bft\in\R_+^{N_\Ome}$ be large and $\delta>0$ be small. Let $D_{\bft,\delta}$ be the set
of points inside $D \setminus H(\bft)$ and at distance $\delta$ from $\partial\Ome$. The
boundary of $D_{\bft,\delta}$ is made of three parts $\Gamma_{1,\bft,\delta}$ in $\partial
D\cap \Ome$, $\Gamma_{2,\bft,\delta}$ in $\partial H(\bft)$ and $\Gamma_{3,\bft,\delta}$
in equidistant curves to $\partial\Ome$. Notice that on $\Gamma_{3,\bft,\delta}$,
$X_u-X_v$ goes to $0$ as $\delta$ goes to $0$ since $X_u=X_v$ on $\partial\Ome$. So
integrating $\Div(X_u-X_v)$ on $D_{\bft,\delta}$, we obtain
$$
0=\int_{\Gamma_{1,\bft,\delta}}(X_u-X_v)\cdot \nu+
\int_{\Gamma_{2,\bft,\delta}}(X_u-X_v)\cdot \nu +\int_{\Gamma_{3,\bft,\delta}}(X_u-X_v)\cdot
\nu
$$
As $\delta\to 0$, the last term goes to $0$. So, with $\Gamma_{1,\bft}=\partial D\cap
(\Ome\setminus H(\bft))$, we have
$$
\int_{\Gamma_{1,\bft}}(X_u-X_v)\cdot \nu\ge -2\ell(\partial H(\bft)\cap\Ome)
$$
Letting $\bft\to\infty$, we obtain $\int_{\partial D\cap \Ome}(X_u-X_v)\cdot \nu\ge 0$
which contradicts $X_u-X_v$ points inside along $\partial D\cap \Ome$.

\begin{remarq}
In the following, a ideal domain $\Ome$ with a $a/b$ labeling that satisfies the conditions of
Theorem~\ref{th:js} will be called a Jenkins-Serrin domain. A solution $u$ to the
Jenkins-Serrin-Dirichlet problem on $\Ome$ will be called a Jenkins-Serrin solution.
\end{remarq}


\subsection{An example}\label{sec:example}

In this section, we give an example of a Jenkins-Serrin domain $\Ome$ in $\Sigma$.

Let $E_1, \dots,E_p$ be the non parabolic ends of $\Sigma$. We recall that $E_i$ is seen
as the quotient of some $C_c$ by a hyperbolic translation $T$.

Let $l$ be an even integer and $T_l$ be the hyperbolic translation such that ${T_l}^l=T$. Let $p\in
\partial_\infty E_i$. Since $C_c$ is invariant by $T_l$, $T_l$ acts on $E_i$ by isometry. Let us define
$p_j^i={T_l}^{j-1}(p)$ for $1\le j \le l$. Let $t>0$ be large and $H(t)$ be a horodisk at $p$
contained in $E_i$. We define $H_j^i(t)=T_l^{j-1}(H(t))$. There is a value $t_l$ of $t$
such that $H_j^i(t_l)$ is tangent to $H_{j+1}^i(t_l)$. Now we choose $l_i$ even such that
$H_j^i(t_{l_i})\subset E_i$.

Let us consider the ideal domain $\Ome$ whose vertices are the cusp end-points of $\Sigma$ and
the $p_j^i$ for $1\le j\le l_i$ and $1\le i\le q$ and the edges are the geodesics
joining $p_j^i$ to $p_{j+1}^i$ and passing by the tangency point between $H_j^i(t_l)$ and
$H_{j+1}^i(t_l)$.

Let us fix a $a/b$ labeling on $\partial\Ome$, then $\Ome$ is a Jenkins-Serrin domain. In
order to verify the conditions, we choose the horodisks
$H_j^i(t_{l_i})$. For this choice of $\bft$, all the edges of $\Ome$ are contained in the horodisks
so for any inscribed polygonal domain $\boP$ we have $\alpha(\bft)=0=\beta(\bft)$ and the
condition $\alpha=\beta$ is satisfied for $\boP=\Ome$. If $\boP\neq \Ome$,
$C_\boP(\bft)\neq \emptyset$ and then $\gamma(\bft)>0$ which gives $\gamma-2\alpha>0$ and
$\gamma-2\beta>0$.


\section{Construction of harmonic diffeomorphisms}\label{sec:construc}

The aim of this section is to prove the following theorem

\begin{thm}\label{th:harm}
Let $\Sigma$ be an orientable complete hyperbolic surface with finite topology. Then there is a
function $u$ defined on $\Sigma$ solution of the~\eqref{eq:mse} whose graph in $\Sigma\times\R$ has
parabolic conformal type.

 As a consequence, there is a parabolic surface $\Sigma'$ and a harmonic diffeomorphism
 $X: \Sigma'\to \Sigma$. 
\end{thm}


\subsection{The conformal type}

The first point is that we can control the conformal type of the graph of a Jenkins-Serrin solution .

\begin{prop}\label{prop:conftyp}
The graph of a Jenkins-Serrin solution is parabolic.
\end{prop}

\begin{proof}
Actually we are going to prove that each annular end is parabolic.

Let $\Ome$ and $u$ be a Jenkins-Serrin domain and $u$ a Jenkins-Serrin solution.
The annular ends of the graph $G_u$ are given by the parts of $G_u$ above the cusp ends of
$\Sigma$ and the ones above non parabolic annular ends.

The annular ends of $G_u$ above cusp ends are parabolic by Corollary~\ref{cor:cusp}.

So let us consider $E$ a non parabolic end of $\Sigma$. The curve $\partial E$ is
contained in $\Ome$ and in the homotopy class of $\partial_\infty E$ in
$\barre\Sigma^\infty$. Then $\partial E$ bounds an annular
connected component $D\in \Ome$ whose other boundary components are the edges of
$\Ome$ with end points in $\partial_\infty E$. Let $G_E$ denote the graph of $u$
above $D\cup \partial E$. We are going to use that $G_E$ is area minimizing in $\barre D\times\R$
to prove that $G_E$ has quadratic area growth. Thus the annular ends will be parabolic (see \cite{Gri}).

Let $\bft\in \R_+^N$ be large such that $\partial E\cap H(\bft)=\emptyset$. For $r>0$, let $\bft+r$ be
the $N_\Ome$-tuple with all coordinates increased by $r$. Let us notice that $D\setminus
H(\bft+r)$ contains all points in $D$ at distance less than $r$ from $\partial E$. Besides the
boundary of $D\setminus H(\bft+r)$ is made of $\partial E$, a finite number of geodesic arcs whose
lengths are bounded by $r+a_0$ for some constant $a_0>0$ and subarcs of horocycles
whose lengths go to $0$ as $r\to +\infty$.

We define $M=\sup_{\partial E}|u|$. Let $G_E(r)$ denote the part of $G_E$ contained in
$\barre{D\setminus H(\bft+r)}\times [-M-r,M+r]$. $G_E(r)$ contains all points in $G_E$ at intrinsic
distance less than $r$ from its boundary. Let $B(r)$ be the component of
$(\barre{D\setminus H(\bft+r)}\times [-M-r,M+r])\setminus G_E(r)$ contained below
$G_E(r)$. Let $S(r)=\partial B(r)\setminus G_E(r)$, $S(r)$ is a surface in $\barre D\times
\R$ with the same boundary as $G_E(r)$ so since $G_E$ is area minimizing $Area(G_E(r))\le
Area (S(r))$. To estimate the area of $S(r)$, we just say that $S(r)\subset
\partial(\barre{D\setminus H(\bft+r)}\times [-M-r,M+r])$. Since $D$ has finite area and
$\partial (D\setminus H(\bft+r))$ has linear growth, we conclude that $Area(S(r))$ has
quadratic growth. So $G_E$ has quadratic area growth and is parabolic.
\end{proof}

\begin{remarq}
The arguments used in \cite{CoRo} are different from the above ones. They give a precise description of the asymptotic behaviour of the graph of a Jenkins-Serrin solution. 
\end{remarq}


\subsection{Extension}

Here we explain how a Jenkins-Serrin domain can be "extended" to an other Jenkins-Serrin
domain such that the solutions given by Theorem~\ref{th:js} are close on the original
domain.

Let us fix a Jenkins-Serrin domain $\Ome_0$ in $\Sigma$ and let $\gamma_1$, $\gamma_2$ be two consecutive
edges of $\Ome_0$ with $\gamma_1$ labeled $b$ and $\gamma_2$ labeled $a$. The connected
component $D_i$ of $\Sigma\setminus \gamma_i$ that does not contain $\Ome_0$ is isometric to a
hyperbolic half-space. 

Let $E$ be the annular end of $\Sigma$ that contains $\gamma_1$ and $\gamma_2$. Let
$\beta_i$ be the geodesic ray contained in $E$ which is orthogonal to $\gamma_i$ and
$\partial E$. Using the disk model for $\H^2$, let $P$ be the hyperbolic
halfspace bounded by the geodesic $\gamma$ joining $-1$ to $-i$ and containing the origin
and let $\beta$ be the geodesic joining $e^{i\frac\pi4}$ to $e^{-i3\frac\pi4}$. Let
$\phi_i:P\to D_i$ be an
isometry preserving the orientation such that $\phi_i(\gamma)=\gamma_i$ and $\phi_i(\beta\cap
P)=\beta_i\cap D_i$. 

For $t\in[0,\pi/4]$, let $R_t$ be the ideal rhombus in $P$ with vertices
$1$, $ie^{it}$, $-1$ and $-i$ and $R'_t$ be the ideal rhombus with vertices $e^{-it}$,
$i$, $-1$ and $-i$. In $\Sigma$, we define $R_{1,t}=\phi_1(R_t)$ and
$R_{2,t}=\phi_2(R_t')$. We then consider the new ideal
domain $\boD_t=\Ome_0\cup (R_{1,t}\cup R_{2,t})\cup(\gamma_1\cup\gamma_2)$. The $a/b$ labeling of
$\partial\Ome_0$ induces a natural $a/b$ labeling of $\partial \boD_t$ that we consider in
the following. In order to lighten the notation, we define $\Ome=\boD_0$, $R_1=R_{1,0}$
and $R_2=R_{2,0}$ (see Figure~\ref{fig:exten}).

\begin{figure}[h]
\begin{center}
\resizebox{0.9\linewidth}{!}{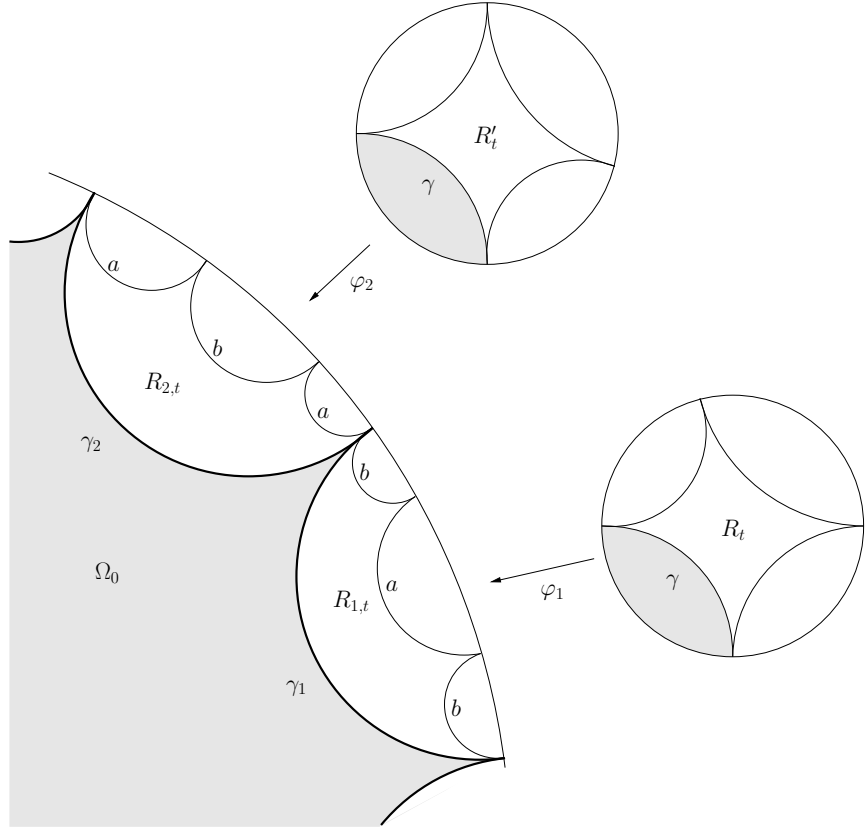}
\caption{The extension of the domain $\Ome_0$ \label{fig:exten}}
\end{center}
\end{figure}

Our aim in this section is to prove the following result

\begin{prop}\label{prop:ext}
Let $\Ome_0$ be a Jenkins-Serrin domain in $\Sigma$, $p\in \Ome_0$ be a point and
$\gamma_1$, $\gamma_2$ be two consecutive edges of $\Ome_0$ with $\gamma_1$ labeled $b$
and $\gamma_2$ labeled $a$. Let $\boD_t$ be the ideal domain defined above. Then for $t>0$
small enough, $\boD_t$ is a Jenkins-Serrin domain.

Let $u$ be the Jenkins-Serrin solution on $\Ome_0$ and $u_t$ be the one on
$\boD_t$ with $u(p)=0=u_t(p)$. Let $K$ be a compact subset in $\Ome_0$ and $\eps$ be a
positive number. Then for $t$ small enough, $\|u-u_t\|_{C^2(K)}\le \eps$.
\end{prop}

The first step consists in analyzing the Jenkins-Serrin conditions on $\Ome=\boD_0$.

\begin{lem}\label{lem:js}
Let $\boP$ be a polygonal domain inscribed in $\Ome$. If $\boP$ is not $R_1$,$R_2$,
$\Ome\setminus\barre R_1$ or $\Ome\setminus \barre R_2$, the Jenkins-Serrin conditions are
satisfied. For $\boP=R_1$ or $\boP=\Ome\setminus \barre R_2$, we have $\gamma-2\beta=0$ and,
for $\boP=R_2$ or $\boP=\Ome\setminus \barre R_1$, we have $\gamma-2\alpha=0$.
\end{lem}

\begin{proof} 
Let $\boP$ be a polygonal domains inscribed in $\Ome$. Since $\Ome_0$ satisfies the
Jenkins-Serrin conditions and $R_1$ and $R_2$ are isometric to ideal squares,
$\boP=R_1,R_2,\Ome\setminus \barre R_1$ or $\Ome\setminus \barre R_2$ satisfy the stated
conditions. Moreover, if $\boP=\Ome$, $\alpha-\beta=0$.

Assume now that $\boP$ is not one of these five polygonal domains. By
Remark~\ref{rm:alternate}, we assume that the $a$-components alternate along
$\partial\boP$ (the other case is similar). Let us first notice that, if $\gamma_2\subset
\partial\boP$, then $\boP=R_2$ which is excluded and, if $\gamma_1\subset \partial\boP$,
then $\boP\cap R_1=\emptyset$. Let us introduce some
notations
\begin{itemize}
\item $A_\boP^0=A_\boP\cap \partial \Ome_0$, $B_\boP^0=B_\boP\cap \partial \Ome_0$, $C_\boP^0=C_\boP\cap (\Ome_0\cup \gamma_1)$,
\item $A_\boP^i=A_\boP\cap \partial R_i$, $B_\boP^i=B_\boP\cap \partial R_i$, $C_\boP^i=C_\boP\cap R_i$,
\item $d_i=\gamma_i\cap \boP$.
\end{itemize}
For $\bft$ large, we then define
\begin{itemize}
\item $\alpha^i(\bft)=\ell(A_\boP^i\setminus H(\bft))$
 and $\gamma^i(\bft)=\ell((A_\boP^i\cup
B_\boP^i\cup C_\boP^i)\setminus H(\bft))$, for $i\in\{1,2,3\}$ and
\item $\delta^i(\bft)=\ell(d_i\setminus H(\bft))$.
\end{itemize}
We have $\alpha(\bft)=\alpha^0(\bft)+\alpha^1(\bft)+\alpha^2(\bft)$ and
$\gamma(\bft)=\gamma^0(\bft)+\gamma^1(\bft)+\gamma^2(\bft)$. So we can compute
$\gamma(\bft)-2\alpha(\bft)= K^0(\bft)+K^1(\bft)+K^2(\bft)$ where
\begin{align*}
K^0(\bft)&=\gamma^0(\bft)+\delta^1(\bft)+\delta^2(\bft)-2(\alpha^0(\bft)+\delta^2(\bft)),\\
K^1(\bft)&=\gamma^1(\bft)+\delta^1(\bft)-2(\alpha^1(\bft)+\delta^1(\bft)),\\
K^2(\bft)&=\gamma^2(\bft)+\delta^2(\bft)-2\alpha^2(\bft).
\end{align*}
Actually $K^0(\bft)$ (resp. $K^i(\bft)$) computes $\gamma-2\alpha$ for
$\boP\cap\Ome_0$ (resp. $\boP\cap R_i$) in $\Ome_0$ (resp. $R_i)$. Since $\Ome_0$, $R_1$
and $R_2$ are Jenkins-Serrin domains, these three terms are non-negative (see
Section~\ref{sec:necessar} and Remark~\ref{rem:uniform}). Moreover, since
the $a$-components alternate along $\partial\boP$ and $\boP$ is not $\Ome$, $R_1$, $R_2$,
$\Ome\setminus \barre R_1$ or $\Ome\setminus \barre R_2$, $C_\boP^0$ is not equal to $\gamma_1$. This
implies that $\boP\cap \Ome_0\neq \Ome_0$ and $K^0(\bft)>0$ for large $\bft$. Thus the
condition $\gamma-2\alpha>0$ is proved for $\boP$.
\end{proof}

The second step of the extension argument consists in proving the first statement of
Proposition~\ref{prop:ext}. We notice that the family $\{\boD_t\}_t$ is a continuous
family of ideal domains in $\Sigma$ : $\Ome=\boD_0$ is
$\bigcup_{t>0}(\textrm{interior}\bigcap_{0<s<t}\boD_s)$. 

\begin{lem}
For $t>0$ small enough, $\boD_t$ is a Jenkins-Serrin domain.
\end{lem}

\begin{proof}
If $\boP$ is an inscribed polygonal domain in $\boD_{t_0}$, we can actually define a
unique continuous family $\{\boP_t\}_t$ such that $\boP_t$ is an inscribed polygonal domain
in $\boD_t$ such that $\boP=\boP_{t_0}$. Assume that the $a$-edges alternate along
$\partial\boP_t$ and $\boP_t$ is not $\boD_t$, $R_{2,t}$ or $\boD_t\setminus \barre
R_{1,t}$. We have several cases to study. If $\boP_t\subset\Ome_0$ then $\boP_t=\boP_0$
and the condition $\gamma-2\alpha>0$ is satisfied. 

We notice that $R_{1,t}$ and $R_{2,t}$ are isometric through an isometry $S$ that send $\gamma_1$ to $\gamma_2$ and exchanges the labels of the edges. 
If $R_{1,t}\cup R_{2,t}\subset \boP_t$, we have $R_{1,s}\cup
R_{2,s}\subset \boP_s$ and $\boP_t\cap\Ome_0=\boP_s\cap\Ome_0$ for any $s\le t$. The
symmetry $S$ implies that the value of
$\gamma-2\alpha$ for $\boP_s$ does not depend on $s$ so $\gamma-2\alpha>0$ on $\boP_t$.

Let us assume $R_{2,t}\subset \boP_t$ and $R_{1,t}\cap \boP_t=\emptyset$. A decomposition similar
to the one in Lemma~\ref{lem:js} proof gives $\gamma(\bft)-2\alpha(\bft)=K^0(\bft)+K^2(\bft)$
where $K^0(\bft)>0$ since $\boP_t\cap\Ome_0$ does not depend on $t$ and $K^2(\bft)>0$ for
$t>0$ because $R_{2,t}\cap \boP_t=R_{2,t}$ which is isometric to $R'_t$ (see
Figure~\ref{fig:Rt}). So $\gamma-2\alpha>0$ is satisfied for
$\boP_t$.

If we are not in the cases $\boP_t\subset \Ome_0$, $R_{1,t}\cup R_{2,t}\subset \boP_t$
or $R_{2,t}\subset \boP_t$ and $R_{1,t}\cap \boP_t=\emptyset$, we can be sure that
$C_{\boP_t}$ intersects $\gamma_2$ or $\gamma_1$ since the $a$-edges alternate. More
precisely,
there is a compact subset in $\Ome_0$ (close to $\gamma_2$ and $\gamma_1$) that does not depend on the
particular $\boP_t$ such that $C_{\boP_t}\cap K$ contains a subarc of length at least
$\eps$ ($\eps$ independent of $\boP_t$). Let $u_0$ be the Jenkins-Serrin solution on
$\Ome_0$. We have $\|X_{u_0}\|\le 1-\delta$ ($\delta>0$) on $K$. So by
Remark~\ref{rem:uniform}, we have $\gamma-2\alpha\ge \delta\eps>0$ on $\boP_0$. Since the
value of $\gamma-2\alpha$ on $\boP_t$ depends continuously on $t$, we see that
$\gamma-2\alpha>\delta\eps/2$ for $\boP_t$ for any $t\le t_0$ where $t_0$ does not depend
on the particular $\boP_t$.

For $\boP=\boD_t$, $\alpha-\beta=0$ comes from the fact that $\Ome_0$ is a
Jenkins-Serrin domain and $S$ is an isometry from $R_{1,t}$ to $R_{2,t}$ exchanging the
label of the edges.

For $\boP_t=R_{2,t}$, the condition $\gamma-2\alpha$ for $t>0$ can be easily
verified on $R_t'$ so the same is true on $\boP_t=R_{2,t}$ (see Figure~\ref{fig:Rt}).

For $\boP=\boD_t\setminus \barre R_{1,t}$, the condition $\gamma-2\alpha>0$
then follows from the fact that $\Ome_0$ is a Jenkins-Serrin domain and the condition
$\gamma-2\alpha>0$ on $R'_t$.

The same argument can be done for polygonal domains with alternating $b$-edges.
\end{proof}

\begin{figure}[h]
\begin{center}
\resizebox{0.6\linewidth}{!}{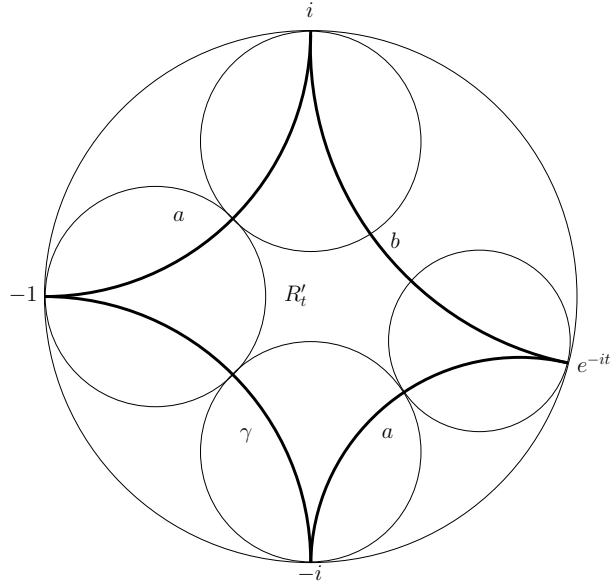}
\caption{The ideal rhombus $R_t'$ with a choice of horodisks proving the $\gamma-2\alpha>0$ condition \label{fig:Rt}}
\end{center}
\end{figure}

From the above result, there is a Jenkins-Serrin solution $u_t$ on
$\boD_t$, we consider the one satisfying $u_t(p)=0$. As $t$ goes to $0$, $\boD_t$ goes to
$\Ome$ and, considering a subsequence, $X_{u_t}$ converges to some $X$ on $\Ome$. The
description of $X$ is given by the following result.

\begin{lem}
$X$ has exactly two divergence lines : the geodesic lines $\gamma_1$ and $\gamma_2$. Along
$\gamma_1$, $X$ points into $\Ome_0$ and along $\gamma_2$, $X$ points into $R_2$. Moreover, $X=\nu$ (resp. $X=-\nu$) along the $a$-boundary
components (resp. $b$-boundary components) of $\partial\Ome$.
\end{lem}

\begin{proof}
First, because of the value of $X_{u_t}$ along $\partial\boD_t$, $X_{u_t}=\nu$ (resp.
$X_{u_t}=-\nu$) along the $a$-boundary components (resp. $b$-boundary components) of
$\partial\boD_t$. As a consequence $X=\nu$ (resp. $X=-\nu$) along the $a$-boundary
components (resp. $b$-boundary components) of $\partial\Ome$ (Lemma~\ref{lem:nodiv2}).

If $X$ has no divergence line, then, considering a subsequence, $u_t$ converges to $u$ a solution
of \eqref{eq:mse} on $\Ome$. Because of the value of $X$ along $\partial\Ome$, $u$ is then
a Jenkins-Serrin solution on $\Ome$ which is impossible since $\Ome$ is not
a Jenkins-Serrin domain. So $X$ must have at least one divergence line.

Moreover the value of $X$ along $\partial \Ome$ implies that the divergences lines are
either closed geodesics or proper geodesics ending at vertices of $\Ome$.

As in the proof of Theorem~\ref{th:js}, we introduce an oriented graph structure $G$
on the set of connected component of $\boB(X)$. Using the same arguments, there is an
inscribed polygonal domain
$\boP$ in
$\Ome$ which is a connected component of $\boB(X)$ where the condition $\gamma-2\alpha>0$ is not
satisfied. So $\boP=R_2$ or $\boP=\Ome\setminus \barre R_1$ by Lemma~\ref{lem:js}. There is also an inscribed
polygonal domain $\boP'$ in $\Ome$ which is a connected component of $\boB(X)$ where the
condition $\gamma-2\beta>0$ is not satisfied. So $\boP'=R_1$ or $\boP'=\Ome\setminus \barre R_2$.

This implies that at least $\gamma_1$ or $\gamma_2$ is a divergence line with the stated
value of $X$ along it. Assume only $\gamma_1$ is a divergence line (the same can be done for
$\gamma_2$). This would imply that $\boB(X)$ has only two connected components $\boP'=R_1$ and
$\boP=\Ome\setminus \barre R_1$. So a subsequence of $u_t$ converges to a solution $u$ of
\eqref{eq:mse} on $\Ome\setminus \barre R_1$. Because of the value of $X$ along $\partial
(\Ome\setminus \barre R_1)$, $u$ would be a  Jenkins-Serrin solution on
$\Ome\setminus \barre R_1$ which is impossible since $\Ome\setminus \barre R_1$ is not a
Jenkins-Serrin domain by Lemma~\ref{lem:js}. $\gamma_1$ and $\gamma_2$ are divergence lines.

The last point consists	in proving there are no other divergence lines. We notice that
$\gamma_1$ and $\gamma_2$ split the graph $G$ in three connected components. If one of
these three components contains one edge, then a similar argument will prove that $\Ome_0$,
$R_1$ or $R_2$ is not a Jenkins-Serrin domain.
\end{proof}



We can now finish the proof of Proposition~\ref{prop:ext}.

\begin{proof}[Proof of Proposition~\ref{prop:ext}]
The structure of $\boB(X)$ implies that on $\Ome_0$, a subsequence of $u_t$ converges to $v$
a solution of \eqref{eq:mse} on $\Ome_0$. Besides the value of $X$ on $\partial \Ome_0$
implies that $v$ is a Jenkins-Serrin solution on $\Ome_0$. By uniqueness of
this solution and since $u(p)=0=v(p)$ we have $u=v$ on $\Ome$. Since this limit does not
depend on the sequence, this implies that $u_t\to u$ uniformly on each compact subset of
$\Ome_0$. So $\|u-u_t\|_{C^2(K)}\le \eps$ for $t$ small enough.
\end{proof}


\subsection{The construction}

Using the preceding results, we are ready to prove TheoremÄ±~\ref{th:harm}.

\begin{prop}\label{prop:seq}
Let $\Ome$ be a Jenkins-Serrin domain. There is an increasing
sequence $(\Ome_n)_n$ of Jenkins-Serrin domains ($\Ome_0=\Ome$) and an increasing sequence
of compact subsets $(K_n)_n$ of $\Ome_n$ such the following is true. Let $u_n$ be the
Jenkins-Serrin solution on $\Ome_n$ and $E_i$ be the annular ends of $\Sigma$; we then
have
\begin{itemize}
\item $K_n\subset \Ome_n$ and $\cup_n K_n=\Sigma$,
\item $\|u_{n+1}-u_n\|_{C^2(K_n)}\le \frac1{2^n}$,
\item $(K_n\setminus \inter K_{n-1})\cap E_i$ is an annulus and
\item the graph of $u_n$ over $(K_j\setminus K_{j-1})\cap E_i$ is an annulus whose conformal modulus is
at least $1$ for any $n\ge j$.
\end{itemize}
\end{prop}

\begin{proof}
Fix a point $\bar p\in\Sigma\setminus \cup_i E_i$, there is a constant $\bar d>0$ such the following
is true. There are sequences $\Ome_n$, $K_n$ and $u_n$ such that 
\begin{itemize}
\item $K_n\subset \Ome_n$, $K_n$ contains $\{p\in \Ome_n| d(p,\partial \Ome_n)\ge 1\}$,
$d(\bar p,\partial \Ome_n)\ge n\bar d$,
\item $\|u_{n+1}-u_n\|_{C^2(K_n)}\le \frac1{2^n}$,
\item $(K_n\setminus \inter K_{n-1})\cap E_i$ is an annulus and
\item the graph of $u_n$ over $(K_j\setminus K_{j-1})\cap E_i$ is an annulus whose conformal modulus is
at least $1$ for any $n\ge j$.
\end{itemize}

Clearly this will prove the proposition. The proof of the existence is by induction. So assume that $\Ome_j$, $K_j$ and $u_j$ are constructed for $j\le n$

Since $\Ome_n$ is a Jenkins-Serrin domain, we can gather its edges in a finite number of pairs
$\{\gamma_1^i,\gamma_2^i\}$ such that $\gamma_1^i$, $\gamma_2^i$ are consecutive edges and
$\gamma_1^i$ is labeled $b$ and $\gamma_2^i$ is labeled $a$. Let $\eps$ be positive and apply
Proposition~\ref{prop:ext} successively to the pair {$\gamma_1^i$, $\gamma_2^i$} to add
perturbed squares along these edges. We obtain a Jenkins-Serrin domain $\Ome_{n+1}$
and a solution $u_{n+1}$ such that $\|u_n-u_{n+1}\|_{C^2(K_n)}\le \eps$. Choosing $\eps$
sufficiently small, we can ensure that $\|u_n-u_{n+1}\|_{C^2(K_n)}\le \frac 1{2^n}$ and
the graph of $u_{n+1}$ over $(K_j\setminus \inter K_{j-1})\cap E_i$ is an annulus whose conformal
modulus is at least $1$ for $j\le n$.

Next we can choose a compact subset $K_{n+1}$ containing $\{p\in \Ome_{n+1}| d(p,\partial
\Ome_{n+1})\le 1\}$ such that $K_{n+1}\setminus \inter K_n$ is made of annuli in each
annular end of $\Sigma$. Moreover, by Proposition~\ref{prop:conftyp} the graph of
$u_{n+1}$ is parabolic. So $K_{n+1}$ can be chosen such that the graph of
$u_{n+1}$ over the annuli of $K_{n+1}\setminus \inter K_n$ has conformal modulus at least
$1$.

The last point we have to check is that $d(\bar p,\partial\Ome_{n+1})\ge (n+1)\bar d$. For this we just
have to analyze how the distance between $\partial E_i$ and $\partial \Ome_n$
in $E_i$ evolves. Actually Lemma~\ref{lem:exaust} (see also Figure~\ref{fig:dist}) implies that, in $E_i$, $d(\partial E_i,\partial
\Ome_n)\ge n d_{\kappa_i}$ where $\kappa_i$ is the curvature of $\partial E_i$; thus
$d(\bar p,\partial\Ome_{n+1})\ge (n+1)\bar d$ if $\bar d=\min d_{\kappa_i}$.
\end{proof}

We can now prove our main theorem.

\begin{proof}[Proof of Theorem~\ref{th:harm}]
Starting with the Jenkins-Serrin domain given in Section~\ref{sec:example}, we apply
Proposition~\ref{prop:seq} to construct $\Ome_n$, $K_n$ and $u_n$. Since $\cup_n
K_n=\Sigma$ and $\|u_n-u_{n+1}\|_{C^2(K_n)}\le \frac1{2^n}$, $u_n$ converges to a solution
$u$ of \eqref{eq:mse} on $\Sigma$, the convergence is smooth on any compact subsets of
$\Sigma$.

Since $(u_n)$ converges smoothly to $u$, the modulus of the graph of $u$ on each annular
component of $K_i\setminus\inter K_{i-1}$ is at least $1$. This implies that each annular
end of the graph of $u$ has infinite conformal modulus and is parabolic. Thus the graph of $u$ is parabolic.
\end{proof}


\appendix
\section{An equicontinuity result}

Let us fix some notations. If $p\in \R^n$, $N\in\S^{n-1}$ and $\delta>0$, we denote by $D(p,N,\delta)$ the ball in the hyperplane passing through $p$ and normal to $N$  with center $p$ and radius $\delta$. Then we denote by $C(p,N,\delta)$ the cylinder $\{q+sN, q\in D(p,N,\delta) \textrm{ and }s\in \R\}$. Finally if $S$ is a hypersurface in $\R^n$ and $p\in S$ and $N(p)$ denote the unit normal to $S$, we denote by $S(p,\delta)$ the connected component of $S\cap C(p,N(p),\delta)$ containing $p$.

We first begin by recalling a classical result (see for example, Lemma 2.4 in \cite{CoMi10} or Lemma 4.1.1 in \cite{PeRo}).

\begin{prop}\label{prop:loc}
Let $c$ and $\delta$ be positive, there is $\delta'>0$ such the following is true. Let $S$ be a hypersurface in $\R^n$ and $p\in S$ such that 
the second fundamental form of $S$ is bounded by $c$ and $d_S(p,\partial S)\ge \delta$.
Then $S(p,\delta')$ is a graph over $D(p,N(p),\delta')$.
Moreover, the function $v$ which defines this graph satisfies $v(q)\le 8c|p-q|^2$,
$|\nabla v(q)|\le 8c|p-q|$ and $|\nabla^2 v|\le 16c$ for any $q\in D(p,\delta')$.
\end{prop}

A consequence of this local description is the following result.

\begin{prop}\label{prop:equi1}
Let $U\subset \Ome$ be two open subsets of $\R^{n-1}$ and $c$, $\delta$ be positive. Let
$\boS$ be a set of smooth functions on $\Ome$ such that, for any $p\in U$ and $u\in \boS$,
the following is true
\begin{itemize}
\item $d_{G_u}((p,u(p)),\partial G_u)\ge \delta$ and
\item the second fundamental form of $G_u$ is bounded by $c$ on the geodesic disk of
radius $\delta$ and center $(p,u(p))$. 
\end{itemize}
Then the family $\{X_u : U\to \R^{n-1}, u\in\boS\}$ is  uniformly equicontinuous where $U$ is endowed
with the geodesic metric.
\end{prop}

Let us recall that the geodesic distance $d_\gamma$ between two points in $U$ is given by the infimum
of the length of curves in $U$ joining the two points. In an open set, $d_\gamma$ induces
the usual topology.

\begin{proof}
Proving $X_u$ is uniformly equicontinuous is the same as proving $N_u$ is uniformly
equicontinuous. So if $\{N_u : U\to \S^{n-1}, u\in\boS\}$ is not uniformly equicontinous,
it means that we have two sequences $(p_{1,n})_n$ and $(p_{2,n})_n$ in $U$ and a sequence
$u_n$ in $\boS$ such that $d_\gamma(p_{1,n},p_{2,n})\to 0$, $N_{u_n}(p_{1,n})\to
N_1$ and $N_{u_n}(p_{2,n})\to N_2$ with $N_1\neq N_2$. Moreover, by changing
the point $p_{2,n}$ by a point along a curve of length at most
$d_{\gamma}(p_{1,n},p_{2,n})+\frac1n$ between $p_{1,n}$ and $p_{2,n}$, we can assume
$d_{\S^{n-1}}(N_{u_n}(p_{1,n}),N_{u_n}(p_{2,n}))\le \pi/2$ and so
$\alpha=d_{\S^{n-1}}(N_1,N_2)\le \pi/2$. By Proposition~\ref{prop:loc}, there is $\delta'$
such that on $G_{u_n}(p_{1,n},\delta')$ the unit normal is at distance less that
$\alpha/3$ from $N_{u_n}(p_n)$. Moreover, the translate
$G_{u_n}(p_{1,n},\delta'/2)-p_{1,n}-u_n(p_{1,n})\partial_{x_n}$ converges (after taking a
subsequence) in $C^1$ topology to a graph over $D(0, N_1,\delta'/2)$ along which the unit
normal is at distance less than $\alpha/3$ from $N_1$. By the same argument,
$G_{u_n}(p_{2,n},\delta'/2)-p_{2,n}-u_n(p_{2,n})\partial_{x_n}$ converges in $C^1$
topology to a graph over $D(0, N_2,\delta'/2)$ along which the unit normal is at distance
less than $\alpha/3$ from $N_2$. Since $0<d_{\S^{n-1}}(N_1,N_2)<\pi/2$, these two limit
graphs intersect and are transverse. Thus
$G_{u_n}(p_{1,n},\delta'/2)-u_n(p_{1,n})\partial_{x_n}$ and
$G_{u_n}(p_{2,n},\delta'/2)-u_n(p_{2,n})\partial_{x_n}$ must intersect and be transverse
for $n$ large. This is impossible since at an intersection point the normals have to be the
same; indeed, these two surfaces are vertical translates of the same graph.
\end{proof}

In this paper, this has the following consequence.

\begin{prop}\label{prop:equi2}
Let $U\subset\Ome\subset \Sigma$ be two open subsets of a Riemannian surface ($U$ with compact closure). Let $\delta$
be positive. Let $\boS$ be a set of solutions of \eqref{eq:mse} on $\Ome$ such that for any
$p\in U$ and $u\in \boS$, $d_{G_u}((p,u(p)),\partial G_u)\ge \delta$. Then the family
$\{X_u : U\to T\Sigma, u\in\boS\}$ is equicontinuous.
\end{prop}

\begin{proof}
Let $p\in U$ and consider a local chart $\phi:V\in \R^2\to U$ around $p$. Because of the hypothesis
$d_{G_u}((p,u(p)),\partial G_u)\ge \delta$, curvature estimates for stable minimal
surfaces \cite{RoSoTo} apply to prove that $G_u$ has uniformly bounded second fundamental form near
$(p,u(p))$ in $\Sigma\times\R$. This implies that, in $\R^2$, the family $\{u\circ \phi,
u\in \boS\}$ satisfies the hypotheses of Proposition~\ref{prop:equi1} for some open set
$W\subset V$ ($\phi^{-1}(p)\in W$). This implies the equicontinuity of $X_{u\circ\phi}$ at
$\phi^{-1}(p)$ and then the one of $X_u$ at $p$.
\end{proof}

In the above proposition, if $U\subset\subset \Ome$, the property
$d_{G_u}((p,u(p)),\partial G_u)\ge \delta$ is satisfied, so $\{X_u : \Ome \to T\Sigma\}$
is equicontinuous.


\section{A technical lemma}

In this section, we prove the following result (see Figure~\ref{fig:dist}).

\begin{lem}\label{lem:exaust}
Let $\kappa$ be in $(0,1)$ then there is a positive constant $d_\kappa$ such that the
following is true. Let $c$ be a complete curve of constant curvature $\kappa$ in $\H^2$
and $\gamma$ be a complete geodesic contained in the non-meanconvex side of $c$. Let
$\gamma'$ be the unique geodesic orthogonal to $\gamma$ and $c$. In the halfplane bounded
by $\gamma$ that
does not contain $c$, there are two uniquely determined geodesic rays $\gamma_1$ and
$\gamma_2$ starting from a common point in $p\in\gamma'$ and ending at the endpoints of
$\gamma$ that are orthogonal at $p$. Then we have
$$
d(c,\gamma_1\cup\gamma_2)\ge d(c,\gamma)+d_\kappa.
$$
\end{lem}

\begin{proof}
Let us in fact consider the foliation of $\H^2$ by curves $\{c_t\}_{t=\in\R}$ of constant
curvature $\kappa$ and orthogonal to $\gamma'$ such that $d(c_t,c_{t'})=|t-t'|$,
$c=c_{-d(c,\gamma)}$ and $c_o$ is tangent to $\gamma$. There is some $d_\kappa>0$ such
that $c_{d_\kappa}$ is tangent to $\gamma_1$ and $\gamma_2$ and contained in the quarter
space bouded by $\gamma_1$ and $\gamma_2$ (notice that $d_k$ only depends on $\kappa$). We
then have
$$
d(c,\gamma_1\cup\gamma_2)\ge d(c,c_{d_\kappa})= d(c,\gamma)+d_\kappa.
$$
\end{proof}
\begin{figure}[h]
\begin{center}
\resizebox{0.6\linewidth}{!}{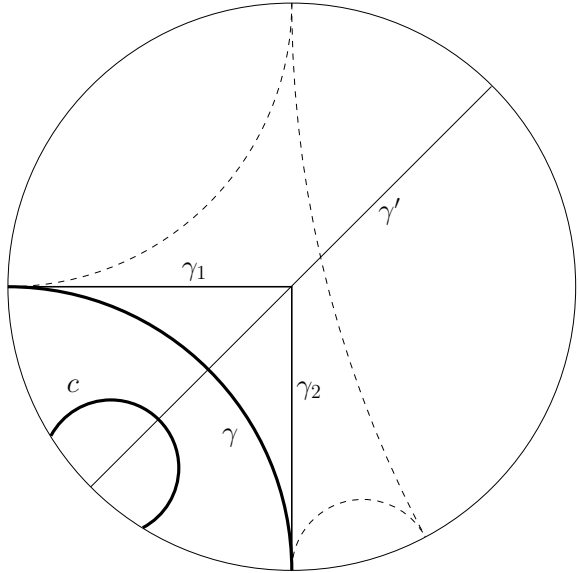}
\caption{The geodesic rays $\gamma_1$ and $\gamma_2$ with $R_t'$ drawn \label{fig:dist}}
\end{center}
\end{figure}


\end{document}

%% file: fig2.pdf_t
\begin{picture}(0,0)%
\includegraphics{fig2.pdf}%
\end{picture}%
\setlength{\unitlength}{4144sp}%
\begingroup\makeatletter\ifx\SetFigFont\undefined%
\gdef\SetFigFont#1#2#3#4#5{%
  \reset@font\fontsize{#1}{#2pt}%
  \fontfamily{#3}\fontseries{#4}\fontshape{#5}%
  \selectfont}%
\fi\endgroup%
\begin{picture}(8886,5981)(3838,-6954)
\put(7669,-4995){\makebox(0,0)[lb]{\smash{{\SetFigFont{14}{16.8}{\rmdefault}{\mddefault}{\updefault}{\color[rgb]{0,0,0}$\Ome$}%
}}}}
\put(3938,-1961){\makebox(0,0)[lb]{\smash{{\SetFigFont{14}{16.8}{\rmdefault}{\mddefault}{\updefault}{\color[rgb]{0,0,0}$p_2$}%
}}}}
\put(12590,-1378){\makebox(0,0)[lb]{\smash{{\SetFigFont{14}{16.8}{\rmdefault}{\mddefault}{\updefault}{\color[rgb]{0,0,0}$p_3$}%
}}}}
\put(8730,-6858){\makebox(0,0)[lb]{\smash{{\SetFigFont{14}{16.8}{\rmdefault}{\mddefault}{\updefault}{\color[rgb]{0,0,0}$p_{j+1}^1$}%
}}}}
\put(7560,-6826){\makebox(0,0)[lb]{\smash{{\SetFigFont{14}{16.8}{\rmdefault}{\mddefault}{\updefault}{\color[rgb]{0,0,0}$p_j^1$}%
}}}}
\put(8066,-6240){\makebox(0,0)[lb]{\smash{{\SetFigFont{14}{16.8}{\rmdefault}{\mddefault}{\updefault}{\color[rgb]{0,0,0}$\gamma_j^1$}%
}}}}
\end{picture}%

%% file: fig3_bis.pdf_t
\begin{picture}(0,0)%
\includegraphics{fig3_bis.pdf}%
\end{picture}%
\setlength{\unitlength}{4144sp}%
\begingroup\makeatletter\ifx\SetFigFont\undefined%
\gdef\SetFigFont#1#2#3#4#5{%
  \reset@font\fontsize{#1}{#2pt}%
  \fontfamily{#3}\fontseries{#4}\fontshape{#5}%
  \selectfont}%
\fi\endgroup%
\begin{picture}(11776,11349)(1508,-10502)
\put(7606,-8971){\makebox(0,0)[lb]{\smash{{\SetFigFont{20}{24.0}{\rmdefault}{\mddefault}{\updefault}{\color[rgb]{0,0,0}$b$}%
}}}}
\put(6706,-7261){\makebox(0,0)[lb]{\smash{{\SetFigFont{20}{24.0}{\rmdefault}{\mddefault}{\updefault}{\color[rgb]{0,0,0}$a$}%
}}}}
\put(6346,-5731){\makebox(0,0)[lb]{\smash{{\SetFigFont{20}{24.0}{\rmdefault}{\mddefault}{\updefault}{\color[rgb]{0,0,0}$b$}%
}}}}
\put(4321,-4021){\makebox(0,0)[lb]{\smash{{\SetFigFont{20}{24.0}{\rmdefault}{\mddefault}{\updefault}{\color[rgb]{0,0,0}$b$}%
}}}}
\put(5761,-4876){\makebox(0,0)[lb]{\smash{{\SetFigFont{20}{24.0}{\rmdefault}{\mddefault}{\updefault}{\color[rgb]{0,0,0}$a$}%
}}}}
\put(2881,-2851){\makebox(0,0)[lb]{\smash{{\SetFigFont{20}{24.0}{\rmdefault}{\mddefault}{\updefault}{\color[rgb]{0,0,0}$a$}%
}}}}
\put(6210,-3068){\makebox(0,0)[lb]{\smash{{\SetFigFont{20}{24.0}{\rmdefault}{\mddefault}{\updefault}{\color[rgb]{0,0,0}$\phi_2$}%
}}}}
\put(11318,-6503){\makebox(0,0)[lb]{\smash{{\SetFigFont{20}{24.0}{\rmdefault}{\mddefault}{\updefault}{\color[rgb]{0,0,0}$R_t$}%
}}}}
\put(7913,-1096){\makebox(0,0)[lb]{\smash{{\SetFigFont{20}{24.0}{\rmdefault}{\mddefault}{\updefault}{\color[rgb]{0,0,0}$R_t'$}%
}}}}
\put(8835,-7313){\makebox(0,0)[lb]{\smash{{\SetFigFont{20}{24.0}{\rmdefault}{\mddefault}{\updefault}{\color[rgb]{0,0,0}$\phi_1$}%
}}}}
\put(5978,-7479){\makebox(0,0)[lb]{\smash{{\SetFigFont{20}{24.0}{\rmdefault}{\mddefault}{\updefault}{\color[rgb]{0,0,0}$R_{1,t}$}%
}}}}
\put(2513,-5280){\makebox(0,0)[lb]{\smash{{\SetFigFont{20}{24.0}{\rmdefault}{\mddefault}{\updefault}{\color[rgb]{0,0,0}$\gamma_2$}%
}}}}
\put(3383,-4479){\makebox(0,0)[lb]{\smash{{\SetFigFont{20}{24.0}{\rmdefault}{\mddefault}{\updefault}{\color[rgb]{0,0,0}$R_{2,t}$}%
}}}}
\put(10562,-7172){\makebox(0,0)[lb]{\smash{{\SetFigFont{20}{24.0}{\rmdefault}{\mddefault}{\updefault}{\color[rgb]{0,0,0}$\gamma$}%
}}}}
\put(7202,-1742){\makebox(0,0)[lb]{\smash{{\SetFigFont{20}{24.0}{\rmdefault}{\mddefault}{\updefault}{\color[rgb]{0,0,0}$\gamma$}%
}}}}
\put(2708,-7096){\makebox(0,0)[lb]{\smash{{\SetFigFont{20}{24.0}{\rmdefault}{\mddefault}{\updefault}{\color[rgb]{0,0,0}$\Ome_0$}%
}}}}
\put(5334,-8627){\makebox(0,0)[lb]{\smash{{\SetFigFont{20}{24.0}{\rmdefault}{\mddefault}{\updefault}{\color[rgb]{0,0,0}$\gamma_1$}%
}}}}
\end{picture}%

%% file: fig1_bis.pdf_t
\begin{picture}(0,0)%
\includegraphics{fig1_bis.pdf}%
\end{picture}%
\setlength{\unitlength}{4144sp}%
\begingroup\makeatletter\ifx\SetFigFont\undefined%
\gdef\SetFigFont#1#2#3#4#5{%
  \reset@font\fontsize{#1}{#2pt}%
  \fontfamily{#3}\fontseries{#4}\fontshape{#5}%
  \selectfont}%
\fi\endgroup%
\begin{picture}(5790,6009)(3451,-6853)
\put(9226,-4651){\makebox(0,0)[lb]{\smash{{\SetFigFont{14}{16.8}{\rmdefault}{\mddefault}{\updefault}{\color[rgb]{0,0,0}$e^{-it}$}%
}}}}
\put(3466,-3931){\makebox(0,0)[lb]{\smash{{\SetFigFont{14}{16.8}{\rmdefault}{\mddefault}{\updefault}{\color[rgb]{0,0,0}$-1$}%
}}}}
\put(5806,-5326){\makebox(0,0)[lb]{\smash{{\SetFigFont{14}{16.8}{\rmdefault}{\mddefault}{\updefault}{\color[rgb]{0,0,0}$\gamma$}%
}}}}
\put(7246,-5326){\makebox(0,0)[lb]{\smash{{\SetFigFont{14}{16.8}{\rmdefault}{\mddefault}{\updefault}{\color[rgb]{0,0,0}$a$}%
}}}}
\put(6481,-1051){\makebox(0,0)[lb]{\smash{{\SetFigFont{14}{16.8}{\rmdefault}{\mddefault}{\updefault}{\color[rgb]{0,0,0}$i$}%
}}}}
\put(7336,-3391){\makebox(0,0)[lb]{\smash{{\SetFigFont{14}{16.8}{\rmdefault}{\mddefault}{\updefault}{\color[rgb]{0,0,0}$b$}%
}}}}
\put(5131,-3121){\makebox(0,0)[lb]{\smash{{\SetFigFont{14}{16.8}{\rmdefault}{\mddefault}{\updefault}{\color[rgb]{0,0,0}$a$}%
}}}}
\put(6391,-6766){\makebox(0,0)[lb]{\smash{{\SetFigFont{14}{16.8}{\rmdefault}{\mddefault}{\updefault}{\color[rgb]{0,0,0}$-i$}%
}}}}
\put(6256,-3931){\makebox(0,0)[lb]{\smash{{\SetFigFont{14}{16.8}{\rmdefault}{\mddefault}{\updefault}{\color[rgb]{0,0,0}$R_t'$}%
}}}}
\end{picture}%

%% file: fig4.pdf_t
\begin{picture}(0,0)%
\includegraphics{fig4.pdf}%
\end{picture}%
\setlength{\unitlength}{4144sp}%
\begingroup\makeatletter\ifx\SetFigFont\undefined%
\gdef\SetFigFont#1#2#3#4#5{%
  \reset@font\fontsize{#1}{#2pt}%
  \fontfamily{#3}\fontseries{#4}\fontshape{#5}%
  \selectfont}%
\fi\endgroup%
\begin{picture}(5431,5431)(2678,-5484)
\put(3260,-3765){\makebox(0,0)[lb]{\smash{{\SetFigFont{17}{20.4}{\rmdefault}{\mddefault}{\updefault}{\color[rgb]{0,0,0}$c$}%
}}}}
\put(4740,-4225){\makebox(0,0)[lb]{\smash{{\SetFigFont{17}{20.4}{\rmdefault}{\mddefault}{\updefault}{\color[rgb]{0,0,0}$\gamma$}%
}}}}
\put(6225,-2120){\makebox(0,0)[lb]{\smash{{\SetFigFont{17}{20.4}{\rmdefault}{\mddefault}{\updefault}{\color[rgb]{0,0,0}$\gamma'$}%
}}}}
\put(5445,-3780){\makebox(0,0)[lb]{\smash{{\SetFigFont{17}{20.4}{\rmdefault}{\mddefault}{\updefault}{\color[rgb]{0,0,0}$\gamma_2$}%
}}}}
\put(4355,-2650){\makebox(0,0)[lb]{\smash{{\SetFigFont{17}{20.4}{\rmdefault}{\mddefault}{\updefault}{\color[rgb]{0,0,0}$\gamma_1$}%
}}}}
\end{picture}%